\documentclass[a4paper]{amsart}
\usepackage{amsmath,amssymb}
\usepackage{tikz}
\usepackage{hyperref}

\title[On the Uniform Computational Content of Ramsey's Theorem]
{On the Uniform Computational Content\\ of Ramsey's Theorem}

\author{Vasco Brattka}
\address{Faculty of Computer Science, Universit\"at der Bundeswehr M\"unchen, Germany and 
             Department of Mathematics \& Applied Mathematics, University of Cape Town, South Africa\footnote{The research leading to these results has received funding from the National Research Foundation of South Africa and by the People Programme (Marie Curie Actions) of the European Union's Seventh Framework Programme FP7/2007-2013/ under REA grant agreement no PIRSES-GA-2011-294962-COMPUTAL.}}
\email{Vasco.Brattka@cca-net.de}

\author{Tahina Rakotoniaina}
\address{Department of Mathematics \& Applied Mathematics, University of Cape Town, South Africa\footnote{Tahina Rakotoniaina has additionally been supported by the German Academic Exchange Service (DAAD).}}
\email{fenoira@gmail.com}

\def\AA{{\mathcal A}}

\def\CC{{\mathcal C}}
\def\DD{{\mathcal D}}

\def\HH{{\mathcal H}}

\def\IN{{\mathbb{N}}}

\def\IR{{\mathbb{R}}}

\def\TO{\Longrightarrow}
\def\In{\subseteq}

\def\prefix{\sqsubseteq}

\def\mto{\rightrightarrows}

\def\id{{\rm id}}
\def\pr{{\rm pr}}

\def\dom{{\rm dom}}
\def\range{{\rm range}}

\def\Tr{{\rm Tr}}

\newcommand{\SO}[1]{{{\bf\Sigma}^0_{#1}}}

\newcommand{\sO}[1]{{\Sigma^0_{#1}}}

\newcommand{\pO}[1]{{\Pi^0_{#1}}}

\def\LPO{\text{\rm\sffamily LPO}}
\def\LLPO{\text{\rm\sffamily LLPO}}
\def\WKL{\text{\rm\sffamily WKL}}
\def\RCA{\text{\rm\sffamily RCA}}
\def\ACA{\text{\rm\sffamily ACA}}
\def\ATR{\text{\rm\sffamily ATR}}

\def\IVT{\text{\rm\sffamily IVT}}

\def\BWT{\text{\rm\sffamily BWT}}

\def\B{\text{\rm\sffamily B}}

\def\C{\mbox{\rm\sffamily C}}
\def\ConC{\mbox{\rm\sffamily CC}}
\def\UC{\mbox{\rm\sffamily UC}}

\def\LPO{\mbox{\rm\sffamily LPO}}
\def\LLPO{\mbox{\rm\sffamily LLPO}}

\def\K{\text{\rm\sffamily K}}

\def\L{\text{\rm\sffamily L}}
\def\CL{\text{\rm\sffamily CL}}
\def\KL{\text{\rm\sffamily KL}}

\def\U{\text{\rm\sffamily U}}

\def\J{\text{\rm\sffamily J}}

\def\DNC{\text{\rm\sffamily DNC}}
\def\ACC{\text{\rm\sffamily ACC}}

\def\MLR{\text{\rm\sffamily MLR}}
\def\KL{\text{\rm\sffamily KL}}
\def\WWKL{\text{\rm\sffamily WWKL}}

\def\COH{\text{\rm\sffamily COH}}
\def\PA{\text{\rm\sffamily PA}}
\def\RT{\text{\rm\sffamily RT}}
\def\CRT{\text{\rm\sffamily CRT}}
\def\SRT{\text{\rm\sffamily SRT}}

\def\CSRT{\text{\rm\sffamily CSRT}}
\def\D{\text{\rm\sffamily D}}
\def\I{\text{\rm\sffamily I}}

\def\R{\mbox{\rm\sffamily R}}

\def\leqT{\mathop{\leq_{\mathrm{T}}}}
\def\nleqT{\mathop{\not\leq_{\mathrm{T}}}}

\def\leqW{\mathop{\leq_{\mathrm{W}}}}
\def\equivW{\mathop{\equiv_{\mathrm{W}}}}

\def\leqSW{\mathop{\leq_{\mathrm{sW}}}}

\def\equivSW{\mathop{\equiv_{\mathrm{sW}}}}

\def\nleqW{\mathop{\not\leq_{\mathrm{W}}}}
\def\nleqSW{\mathop{\not\leq_{\mathrm{sW}}}}

\def\lW{\mathop{<_{\mathrm{W}}}}
\def\lSW{\mathop{<_{\mathrm{sW}}}}
\def\nW{\mathop{|_{\mathrm{W}}}}
\def\nSW{\mathop{|_{\mathrm{sW}}}}

\newcommand{\dash}{\mbox{-}}

\date{\today}

\newtheorem{theorem}{Theorem}[section]
\newtheorem{proposition}[theorem]{Proposition}
\newtheorem{lemma}[theorem]{Lemma}
\newtheorem{fact}[theorem]{Fact}

\newtheorem{corollary}[theorem]{Corollary}

\theoremstyle{definition}
\newtheorem{definition}[theorem]{Definition}

\newtheorem{question}[theorem]{Question}

\begin{document}

\begin{abstract}
We study the uniform computational content of Ramsey's theorem in the Weihrauch lattice. 
Our central results provide information on how Ramsey's theorem behaves under product, parallelization and jumps. 
From these results we can derive a number of important properties of Ramsey's theorem. 
For one,  the parallelization of Ramsey's theorem for cardinality $n\geq1$ and an
arbitrary finite number of colors $k\geq2$ is equivalent to the $n$--th jump of weak K\H{o}nig's lemma.
In particular, Ramsey's theorem for cardinality $n\geq1$ is $\SO{n+2}$--measurable in the effective Borel hierarchy,
but not $\SO{n+1}$--measurable. 
Secondly, we obtain interesting lower bounds, for instance the $n$--th jump of weak K\H{o}nig's lemma
is Weihrauch reducible to (the stable version of) Ramsey's theorem of cardinality $n+2$ for $n\geq2$.
We prove that with strictly increasing numbers of colors Ramsey's theorem forms a strictly increasing chain 
in the Weihrauch lattice.
Our study of jumps also shows that certain uniform variants of Ramsey's theorem that are indistinguishable
from a non-uniform perspective play an important role. 
For instance, the colored version of Ramsey's theorem explicitly includes the color of the homogeneous
set as output information, and the jump of this problem (but not the uncolored variant) is equivalent to the
stable version of Ramsey's theorem of the next greater cardinality.
Finally, we briefly discuss the particular case of Ramsey's theorem for pairs, and we provide some new
separation techniques for problems that involve jumps in this context.
In particular, we study uniform results regarding the relation of boundedness and induction problems to Ramsey's theorem,
and we show that there are some significant differences with the non-uniform situation in reverse mathematics. 
\end{abstract}

\maketitle

%\begin{footnotesize}
\setcounter{tocdepth}{1}
\tableofcontents
%\end{footnotesize}

\pagebreak

\section{Introduction}
\label{sec:introduction}

In this paper we study uniform computational properties of Ramsey's theorem for cardinality $n$ and $k$ colors. 
We briefly recall some basic definitions. 

By $[M]^n:=\{A\In M:|A|=n\}$ we denote the set of subsets of $M$ with exactly $n$ elements.
We identify $k$ with the set $\{0,...,k-1\}$ for every $k\in\IN$. We also allow the case $k=\IN$.
Any map $c:[\IN]^n\to k$ with finite range is called a {\em coloring} ({\em of} $[\IN]^n$).
A subset $M\In\IN$ is called {\em homogeneous} ({\em for} $c$) if there is some $i\in k$ such that $c(A)=i$ for every $A\in[M]^n$.
In this situation we write $c(M)=i$, which is understood to imply that $M$ is homogeneous. 
Frank P.\ Ramsey proved the following theorem \cite{Ram30}.

\begin{theorem}[Ramsey's theorem 1930~\cite{Ram30}] 
\label{thm:Ramsey}
For every coloring $c:[\IN]^n\to k$ with $n,k\geq1$
there exists an infinite homogeneous set $M\In\IN$.
\end{theorem}

We will abbreviate Ramsey's theorem for cardinality $n$ and $k$ colors by $\RT_{n,k}$.\footnote{We do not use the more common abbreviation $\RT_k^n$ since we
will use upper indices to indicate the number of jumps or products.}
The computability theoretic study of Ramsey's theorem started when Specker proved that there exists a computable counterexample for Ramsey's theorem for
pairs \cite{Spe71}, which shows that Ramsey's theorem cannot be proved constructively.  

\begin{theorem}[Specker 1969~\cite{Spe71}] 
\label{thm:Specker}
There exists a computable coloring $c:[\IN]^2\to k$ without a computable
infinite homogeneous set $M\In\IN$.
\end{theorem}

Jockusch provided a very simple proof of Specker's theorem, and he improved Specker's result by showing the following \cite{Joc72}.

\begin{theorem}[Jockusch 1972~\cite{Joc72}] 
\label{thm:Jockusch}
For every computable coloring $c:[\IN]^n\to 2$ with $n\geq1$ there exists an infinite homogeneous
set $M\In\IN$ such that $M'\leqT\emptyset^{(n)}$. 
However, there exists a computable coloring $c:[\IN]^n\to2$ for each $n\geq2$ without an 
infinite homogeneous set $M\In\IN$ that is computable in $\emptyset^{(n-1)}$.
\end{theorem}

Another cornerstone in the study of Ramsey's theorem was the cone avoidance theorem (Theorem~\cite[Theorem~2.1]{SS95})
that was originally proved by Seetapun.

\begin{theorem}[Seetapun 1995~\cite{SS95}]
\label{thm:Seetapun}
Let $c:[\IN]^2\to2$ be a coloring that is computable in $B\In\IN$, and let $(C_i)_i$ be a sequence of sets $C_i\In\IN$
such that $C_i\nleqT B$ for all $i\in\IN$. Then there exists an infinite homogeneous set $M$ for $c$ such that $C_i\nleqT M$ for all $i\in\IN$.
\end{theorem}

This theorem was generalized by Cholak, Jockusch and Slaman who proved in particular the following version \cite[Theorem~12.2]{CJS01}.

\begin{theorem}[Cholak, Jockusch and Slaman 2001~\cite{CJS01}]
\label{thm:CJS-lower}
For every computable coloring $c:[\IN]^n\to k$ there exists an infinite homogeneous set $M\In\IN$ such that $\emptyset^{(n)}\nleqT M'$.
\end{theorem}

Cholak, Jockusch and Slaman also improved Jockusch's theorem (Theorem~\ref{thm:Jockusch})  for the case of Ramsey's theorem for pairs \cite{CJS01,CJS09}.

\begin{theorem}[Cholak, Jockusch and Slaman 2001~\cite{CJS01}]
\label{thm:CJS-upper}
For every computable coloring $c:[\IN]^2\to2$ there exists an infinite homogeneous set $M\In\IN$,
which is low$_2$, i.e., such that $M''\leqT\emptyset''$.
\end{theorem}

\begin{figure}[htb]
\begin{tikzpicture}[scale=.5,auto=left,every node/.style={fill=black!15}]
\useasboundingbox (-4,-2) rectangle (22,25);
\def\rvdots{\raisebox{1mm}[\height][\depth]{$\huge\vdots$}};

  % Nodes
   \node (limSSSS) at (0.5,24) {$\lim^{(4)}$};
  \node (WKLSSSS) at (0.5,22.25) {$\WKL^{(4)}\equivSW\widehat{\C_2^{(4)}}$};
  \node (RT4N) at (5.75,22) {$\RT_{4,\IN}$}; 
  \node (RT4P) at (8.5,21.75) {$\RT_{4,+}$};
  \node (P4) at (10.5,21.625) {$...$};
  \node (RT44) at (12.5,21.5) {$\RT_{4,4}$}; 
  \node (RT43) at (15,21.25) {$\RT_{4,3}$}; 
  \node (RT42) at (17.5,21) {$\RT_{4,2}$}; 
  \node (C2SSSS) at (20,20.75) {$\C_{2}^{(4)}$}; 
  \node (limSSS) at (0.5,18.75) {$\lim^{(3)}$};
  \node (WKLSSS) at (0.5,17) {$\WKL^{(3)}\equivSW\widehat{\C_2^{(3)}}$};
  \node (RT3N) at (5.75,16.75) {$\RT_{3,\IN}$}; 
  \node (RT3P) at (8.5,16.5) {$\RT_{3,+}$};
   \node (P3) at (10.5,16.375) {$...$};
  \node (RT34) at (12.5,16.25) {$\RT_{3,4}$}; 
  \node (RT33) at (15,16) {$\RT_{3,3}$}; 
  \node (RT32) at (17.5,15.75) {$\RT_{3,2}$}; 
  \node (C2SSS) at (20,15.5) {$\C_{2}^{(3)}$}; 
  \node (limSS) at (0.5,13.5) {$\lim''$};
  \node (WKLSS) at (0.5,11.75) {$\WKL''\equivSW\widehat{\C_2''}$};
  \node (RT2N) at (5.75,11.5) {$\RT_{2,\IN}$}; 
  \node (RT2P) at (8.5,11.25) {$\RT_{2,+}$};
  \node (P2) at (10.5,11.075) {$...$};
  \node (RT24) at (12.5,11) {$\RT_{2,4}$}; 
  \node (RT23) at (15,10.75) {$\RT_{2,3}$}; 
  \node (RT22) at (17.5,10.5) {$\RT_{2,2}$}; 
  \node (C2SS) at (20,10.25) {$\C_{2}''$}; 
  \node (limS) at (0.5,8.25) {$\lim'$};
  \node (WKLS) at (0.5,6.5) {$\WKL'\equivSW\widehat{\C_2'}$};
  \node (RT1N) at (5.75,6.25) {$\RT_{1,\IN}$};
  \node (RT1P) at (8.5,6) {$\RT_{1,+}$};  
  \node (P1) at (10.5,5.875) {$...$};
  \node (RT14) at (12.5,5.75) {$\RT_{1,4}$}; 
  \node (RT13) at (15,5.5) {$\RT_{1,3}$}; 
  \node (RT12) at (17.5,5.25) {$\RT_{1,2}$}; 
  \node (C2S) at (20,5) {$\C_{2}'$}; 
  \node (lim) at (0.5,3) {$\lim\equivSW\widehat{\C_\IN}$};
  \node (WKL) at (0.5,1.25) {$\WKL\equivSW\widehat{\C_2}$};
  \node (CN) at (5.75,2.75) {$\C_\IN$};
  \node (KN) at (8.5,1) {$\K_\IN\equivSW\C_{2}^*$};
  \node (P0) at (11.1,0.75) {$...$};
  \node (C4) at (12.5,0.5) {$\C_{4}$}; 
  \node (C3) at (15,0.25) {$\C_{3}$}; 
  \node (C2) at (17.5,0) {$\C_{2}$}; 

% Straight lines
  \foreach \from/\to in {
  limSSSS/WKLSSSS,
  WKLSSSS/limSSS,
  RT4N/RT4P,
  RT4P/P4,
  P4/RT44,
  RT44/RT43,
  RT43/RT42,
  RT42/RT3N,
  limSSS/WKLSSS,
  WKLSSS/limSS,
  RT3N/RT3P,
  RT3P/P3,
  P3/RT34,
  RT34/RT33,
  RT33/RT32,
  RT32/RT2N,
  limSS/WKLSS,
  WKLSS/limS,
  RT2N/RT2P,
  RT2P/P2,
  P2/RT24,
  RT24/RT23,
  RT23/RT22,
  RT22/RT1N,
  limS/WKLS,
  WKLS/lim,
  RT1N/RT1P,
  RT1P/P1,
  P1/RT14,
  RT14/RT13,
  RT13/RT12,
  lim/WKL,
  lim/CN,
  WKL/KN,
  CN/KN,
  KN/P0,
  P0/C4,
  C4/C3,
  C3/C2,   
  WKLSSSS/RT4N, 
  WKLSSS/RT3N, 
  WKLSS/RT2N,
  WKLS/RT1N,
  C2S/C2,
  C2SS/C2S,
  C2SSS/C2SS,
  C2SSSS/C2SSS}
  \draw [->,thick] (\from) -- (\to);

 % Straight dashed lines
  \foreach \from/\to in {
  RT12/C2,
  RT13/C3,
  RT14/C4,
  RT1N/CN,
%  RT42/WKLSS,
%  RT32/WKLS,
  RT42/C2SSSS,
  RT32/C2SSS,
  RT22/C2SS,
  RT12/C2S,
  RT1P/KN}
  \draw [->,thick,dashed] (\from) -- (\to);

  %Curved lines
  % \draw [->,thick,looseness=1.2] (BWTN) to [out=360,in=90] (CN);

%Curved dashed lines
\draw [->,thick,dashed,looseness=1.5] (RT32) to [out=195,in=59] (WKLS);
\draw [->,thick,dashed,looseness=1.5] (RT42) to [out=195,in=59] (WKLSS);

\draw  (-3,4) rectangle (21,-1);
\draw  (-3.25,9.25) rectangle (21.25,-1.25);
\draw  (-3.5,14.5) rectangle (21.5,-1.5);
\draw  (-3.75,19.75) rectangle (21.75,-1.75);
\draw  (-4,25) rectangle (22,-2);

\node  [fill=none] at (-3.25,24) {$\SO{6}$};
\node  [fill=none] at (-3,19) {$\SO{5}$};
\node  [fill=none] at (-2.75,13.75) {$\SO{4}$};
\node  [fill=none] at (-2.5,8.5) {$\SO{3}$};
\node  [fill=none] at (-2.25,3.25) {$\SO{2}$};

\end{tikzpicture}
  
\ \\[-0.5cm]
\caption{Ramsey's theorem  for different cardinalities and colors in the Weihrauch lattice: all solid arrows indicate strong Weihrauch reductions against the direction of the arrow, all dashed arrows indicate ordinary Weihrauch reductions.}
\label{fig:diagram-RTnk}
\end{figure}
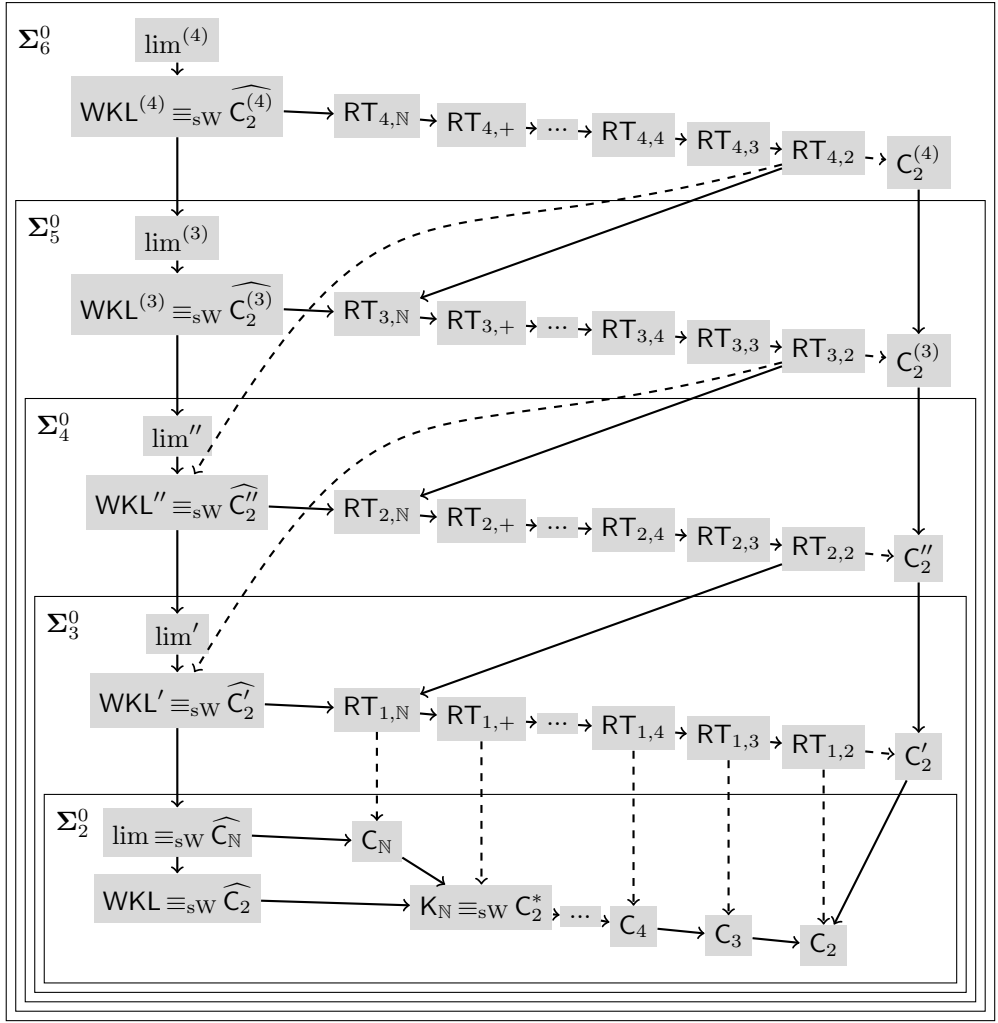

We will make use of these and other earlier results in our uniform study of Ramsey's theorem.
A substantial number of results in this article are based on the second author's PhD thesis \cite{Rak15}.
The first uniform results on Ramsey's theorem were published by Dorais, Dzhafarov, Hirst, Mileti and Shafer~\cite{DDH+16} and by Dzhafarov~\cite{Dzh15,Dzh16}.
Among other things they proved the following squashing theorem~\cite[Theorem~2.5]{DDH+16}
that establishes a relation between products and parallelization for problems such as Ramsey's theorem.

Here a {\em problem} $f:\In\IN^\IN\mto\IN^\IN$, i.e., a partial multi-valued function, is called {\em finitely tolerant} if there is a computable partial function $T:\In\IN^\IN\to\IN^\IN$ such that
for all $p,q\in\dom(f)$ and $k\in\IN$ with $(\forall n\geq k)(p(n)=q(n))$ it follows that $r\in f(q)$ implies $T\langle r,k\rangle\in f(p)$.
By $\langle ,\rangle$ we denote a standard pairing function on Baire space $\IN^\IN$.
Intuitively, finite tolerance means that for two almost identical inputs and a solution for one of these inputs we can compute a solution for the other input. 
The squashing theorem relates products $g\times f$ to parallelizations $\widehat{g}$ of problems (these notions will be defined in the next section).

\begin{theorem}[Squashing theorem \cite{DDH+16}]
\label{thm:squashing}
For $f,g:\In\IN^\IN\mto\IN^\IN$ we obtain:
\begin{enumerate}
\item If $\dom(f)=\IN^\IN$ and $f$ is finitely tolerant, then $g\times f\leqW f\TO \widehat{g}\leqW f$.
\item If $\dom(f)=2^\IN$ and $f$ is finitely tolerant, then $g\times f\leqSW f\TO \widehat{g}\leqSW f$.
\end{enumerate}
\end{theorem} 

This theorem allowed the authors of \cite{DDH+16} to prove that Ramsey's theorem for strictly increasing numbers of colors forms a strictly
increasing chain with respect to {\em strong} Weihrauch reducibility~\cite[Theorem~3.1]{DDH+16}: 
\[\RT_{n,2}\lSW\RT_{n,3}\lSW\RT_{n,3}\lSW....\]
However, they left it an open question~\cite[Question~7.1]{DDH+16} whether an analogous statement also holds for ordinary Weihrauch reducibility.
We will be able to use our Theorem~\ref{thm:products} on products to answer this question in the affirmative (see Theorem~\ref{thm:increasing-colors}).
Independently, similar results were obtained by Hirschfeldt and Jockusch~\cite{HJ16} and Patey~\cite{Pat16a}.
Altogether, the diagram in Figure~\ref{fig:diagram-RTnk} displays how Ramsey's theorem for different cardinalities and colors is situated in the Weihrauch lattice.

We briefly describe the further organization of this article. In the following Section~\ref{sec:preliminaries} we briefly present
the basic concepts and definitions related to the Weihrauch lattice, and we collect a number of facts that are useful for our study.
In Section~\ref{sec:uniform} we introduce the uniform versions of Ramsey's theorem that we are going to consider in this study, and
we establish a core lower bound in Theorem~\ref{thm:lower-bound} that will be used to derive almost all other lower bound results. 
We also prove key results on products in Theorem~\ref{thm:products} and on parallelization in Theorem~\ref{thm:delayed-parallelization}
that also lead to our main lower bound results, which are formulated in Corollary~\ref{cor:discrete-lower-bounds} and Corollary~\ref{cor:delayed-parallelization}.
In Section~\ref{sec:upper} we discuss jumps of Ramsey's theorem together with upper bound results.
In particular, we prove Theorem~\ref{thm:CRT-SRT}, which is our main result on jumps and shows that the stable version of Ramsey's theorem
can be seen as the jump of the colored version of the next smaller cardinality. 
These results lead to Corollary~\ref{cor:induction}, which shows that composition with one jump of weak K\H{o}nig's lemma
is sufficient to bring Ramsey's theorem from one cardinality to the next greater cardinality. 
From this result we can conclude our main upper bound result in Corollary~\ref{cor:upper-bound},
which shows that Ramsey's theorem of cardinality $n$ is reducible to the $n$--th jump of weak K\H{o}nig's lemma.
Together with our lower bound results this finally leads to the classification of the parallelization of Ramsey's theorem of cardinality $n$
as the $n$--th jump of weak K\H{o}nig's lemma in Corollary~\ref{cor:parallelization} and to the classification of the exact Borel degree
of Ramsey's theorem in Corollary~\ref{cor:Borel-measurability}.
We also use these results to obtain the above mentioned result on increasing numbers of colors. 
In Section~\ref{sec:pairs} we briefly discuss the special case of Ramsey's theorem for pairs, 
we summarize some known results, provide some new insights and formulate some open questions. 
In Section~\ref{sec:separation} we develop separation techniques for jumps, and we apply them to
separate some uniform versions of Ramsey's theorem.
Finally, in Section~\ref{sec:boundedness} we discuss the relation of closed and compact choice problems to Ramsey's theorem,
which corresponds to the discussion of boundedness and induction principles in reverse mathematics.

\section{Preliminaries}
\label{sec:preliminaries}

We use the theory of the Weihrauch lattice (as it has been developed in \cite{BBP12,BG11a,BG11,BGH15a,BGM12,BHK15})  
as a framework for the uniform classification of 
the computational content of mathematical problems.
We present a very brief introduction of the main concepts and refer the reader to \cite{BGH15a} and \cite{BHK15}
for a more detailed introduction. 

Formally, the Weihrauch lattice is formed by equivalence classes of partial multi-valued functions $f:\In X\mto Y$
on represented spaces $X,Y$, as defined below. We will simply call such functions {\em problems} here, and they are, in fact, computational
challenges in the sense that for every $x\in\dom(f)$ the goal is to find {\em some} $y\in f(x)$.
In this case $\dom(f)$ contains the admissible instances $x$ of the problem,
and for each instance $x$ the set $f(x)$ contains the corresponding solutions.

For problems $f:\In X\mto Y$, $g:\In Y\mto Z$ we define the {\em composition}   
$g\circ f: \In X\mto Z$ by $g\circ f(x):=\{z\in Z:(\exists y\in f(x))\;z\in g(y)\}$, 
where it is crucial to use the domain $\dom(g\circ f):=\{x\in X:f(x)\In\dom(g)\}$. 
We also denote the composition briefly by $gf$.

A {\em represented space} $(X,\delta)$ is a set $X$ together with a representation, i.e., a surjective partial
map ${\delta:\In\IN^\IN\to X}$ that assigns {\em names} $p\in\IN^\IN$ to points $\delta(p)=x\in X$.
Given a problem $f:\In X\mto Y$ on represented spaces $(X,\delta_X)$ and $(Y,\delta_Y)$ we say
that a partial function $F:\In \IN^\IN\to\IN^\IN$ {\em realizes} $f$, in symbols $F\vdash f$, if $\delta_YF(p)\in f\delta_X(p)$ 
for all $p\in\dom(f\delta_X)$. Notions such as computability and continuity can now be transferred
from Baire space to represented spaces via realizers. For instance, $f$ is called {\em computable}, if it
admits a computable realizer $F$.
We refer the reader to \cite{BHW08,Wei00} for further details.

The intuition behind Weihrauch reducibility is that $f\leqW g$ holds if there is a computational procedure for solving~$f$
during which a single application of the computational resource~$g$ is allowed.
There are actually two slightly different formal versions of this reduction, which are both needed.

\begin{definition}[Weihrauch reducibility]
Let $f: \In X\mto Y$ and $g: \In W\mto Z$ be problems.
\begin{enumerate}
\item 
$f$ is called {\em Weihrauch reducible} to $g$, in symbols $f\leqW g$,
if there are computable ${K: \In X \mto W}$, $H: \In X\times Z\mto Y$ such that
$\emptyset\not=H(x,gK(x))\In f(x)$ for all $x\in\dom(f)$. 
\item $f$ is called {\em strongly Weihrauch reducible} to $g$, in symbols $f\leqSW g$,
if there are computable $K: \In X\mto W$, $H: \In Z\mto Y$ such that $\emptyset\not=HgK(x)\In f(x)$
for all $x\in\dom(f)$.
\end{enumerate}
\end{definition}

In terms of realizers and functions on Baire space (strong) Weihrauch reducibility can also be rephrased as follows \cite[Lemma~4.5]{GM09}:

\begin{itemize}
\item $f\leqW g\iff(\exists$ computable $H,K:\In\IN^\IN\to\IN^\IN)(\forall G\vdash g)\; H\langle \id,GK\rangle\vdash f$.
\item $f\leqSW g\iff(\exists$ computable $H,K:\In\IN^\IN\to\IN^\IN)(\forall G\vdash g)\; HGK\vdash f$.
\end{itemize}

Here $\id:\IN^\IN\to\IN^\IN$ denotes the identity of Baire space.
A relation between Weihrauch reducibility and strong Weihrauch reducibility can be established with the
notion of a cylinder.
A problem $f$ is called a {\em cylinder} if $\id\times f\leqSW f$, which is equivalent
to the property that $g\leqW f\iff g\leqSW f$ holds for all problems $g$ \cite[Corollary~3.6]{BG11}.

Weihrauch reducibility induces a lattice with a rich
and very natural algebraic structure.
We briefly summarize some of these algebraic operations
for problems ${f: \In X\mto Y}$ and $g: \In W\mto Z$:

\begin{itemize}
\item $f\times g:\In X\times W\mto Y\times Z$ is the {\em product} of $f$ and $g$ and represents the parallel evaluation
        of problem $f$ on some input $x$ and $g$ on some input $w$.
\item $f\sqcup g:\In X\sqcup W\mto Y\sqcup Z$ is the {\em coproduct} of $f$ and $g$ and represents the alternative evaluation
        of $f$ on some input $x$ or $g$ on some input $w$ (where $X\sqcup W$ and $Y\sqcup Z$ denote the disjoint unions).
\item $f*g:=\sup\{f_0\circ g_0: f_0\leqW f\mbox{ and }g_0\leqW g\}$ is the {\em compositional product} and represents the consecutive usage of the problem $f$ after the problem $g$.
\item $f^*:=\bigsqcup_{n=0}^\infty f^n$ is the {\em finite parallelization} and allows an evaluation of the $n$--fold product~$f^n$ for some arbitrary given $n\in\IN$.
\item $\widehat{f}:\In X^\IN\mto Y^\IN$ is the {\em parallelization} of $f$ and allows a parallel evaluation of 
        countably many instances of $f$.
\item $f':\In X\mto Y$ denotes the {\em jump} of $f$, which is formally the same problem as $f$, but the input representation $\delta_X$
of $X$ is replaced by its jump $\delta_X':=\delta_X\circ\lim$.
\end{itemize}

Here $\lim: \In\IN^\IN\to\IN^\IN,\langle p_0,p_1,p_2,...\rangle\mapsto\lim_{i\to\infty}p_i$ is the 
limit map on Baire space for $p_i\in\IN^\IN$, where $\langle\ \rangle$ denotes a standard infinite tupling function
on Baire space. 
We recall that $\langle n,k\rangle:=\frac{1}{2}(n+k+1)(n+k)+k$, and inductively this can be 
extended to finite tuples by
$\langle i_{n+1},i_n,...,i_0\rangle:=\langle i_{n+1},\langle i_n,...,i_0\rangle\rangle$
for all $n\geq1$ and $i_0,...,i_{n+1}\in\IN$. Likewise, we define $\langle p_0,p_1,p_2,...\rangle\in\IN^\IN$ for $p_i\in\IN^\IN$ by
$\langle p_0,p_1,p_2,...\rangle\langle n,k\rangle:=p_n(k)$. 
Similarly as $\lim$, we define the limit $\lim_{2^\IN}:\In2^\IN\to2^\IN$ on Cantor space $2^\IN$.
For finite sets $X\In\IN$ we denote by $\lim_X:\In X^\IN\to X$ the ordinary limit operation with respect
to the discrete topology. 

By $f^{(n)}$ we denote the {\em $n$--fold jump} of $f$, and we denote
the {\em $n$--fold compositional product} of $f$ with itself by $f^{[n]}$, i.e., $f^{[0]}=\id$, $f^{[n+1]}:=f^{[n]}*f$.
By $f^n:\In X^n\mto Y^n$ we denote the {\em $n$--fold product} $f\times ...\times f$ of $f$.

One can also define a sum operation $\sqcap$ that play the role
of the infimum for ordinary Weihrauch reducibility $\leqW$. But we are not going to use this operation here. 
Altogether, $\leqW$ induces a lattice structure, and the resulting lattice is called {\em Weihrauch lattice}.
This lattice is not complete as infinite suprema do not need to exist, but the special
supremum $f*g$ is even a maximum that always exists  \cite[Proposition~31]{BP16}, and the compositional product $*$ is associative \cite[Corollary~17]{BP16}.
Finite parallelization and parallelization are closure operators
in the Weihrauch lattice. Further information on the algebraic structure can be found in \cite{BP16}.

For metric spaces $X$ we denote by $\SO{n}$ the corresponding class of Borel subsets of $X$ of level $n\geq1$ \cite{Kec95}.
We recall that a partial function $f:\In X\to Y$ on metric spaces $X$ and $Y$ is called 
{\em $\SO{n}$--measurable}, if for every open set $U\In Y$
there exists a set $V\in\SO{n}$ such that $f^{-1}(U)=V\cap\dom(f)$, i.e., such that
$f^{-1}(U)$ is a $\SO{n}$--set relative to the domain of $f$.
In the case of $n=1$ we obtain exactly continuity. 
Likewise we define $f$ to be {\em effectively $\SO{n}$--measurable} if a corresponding $V$ can be uniformly
computed from a given $U$ \cite[Definition~3.5]{Bra05}. 
In this case we obtain exactly computability for $n=1$. 
The notion of (effective) $\SO{n}$--measurability can be transferred to problems $f:\In X\mto Y$
via realizers, i.e., $f$ is called {\em (effectively) $\SO{n}$--measurable}, if it has a realizer
$F:\In\IN^\IN\to\IN^\IN$ that is (effectively) $\SO{n}$--measurable.\footnote{This definition yields a conservative extension of
the notion of (effective) measurability to multi-valued functions on represented space. It coincides for single-valued functions
on computable Polish spaces with the usual notion of (effective) measurability as it is known from descriptive set theory~\cite{Bra05}.} 
It is easy to see that Weihrauch reducibility preserves (effective) $\SO{n}$--measurability
downwards. 

\begin{fact}
\label{fact:Borel}
Let $f,g$ be problems and $n,k\geq1$. Then
\begin{enumerate}
\item $f$ (effectively) $\SO{n}$--measurable and $g$ (effectively) $\SO{k}$--measurable $\TO f*g$ (effectively) $\SO{n+k-1}$--measurable.
\item $f,g$ (effectively) $\SO{n}$--measurable $\TO f\times g$ (effectively) $\SO{n}$--measurable.
\item $f\leqW g$ and $g$ (effectively) $\SO{n}$--measurable $\TO f$ (effectively) $\SO{n}$--measurable.
\item $f\leqW\lim^{[n]}\iff f$ is effectively $\SO{n+1}$--measurable.
\end{enumerate}
\end{fact}
\begin{proof}
(1) and (2) follow form \cite[Proposition~3.8]{Bra05}.
(3) follows from (1) and (2) 
since $H\langle \id,GK\rangle$ is (effectively) $\SO{n}$--measurable, if $G:\In\IN^\IN\to\IN^\IN$ is so and
the functions
$H,{K:\In\IN^\IN\to\IN^\IN}$ are computable (see also \cite[Proposition~5.2]{Bra05}). 
(4) Here ``$\TO$'' follows from (3) since $\lim$ is effectively $\SO{2}$--measurable \cite[Proposition~9.1]{Bra05},
and hence $\lim^{[n]}$ is effectively $\SO{n+1}$--measurable by (1). 
The inverse implication ``$\Longleftarrow$'' follows with the help of 
the effective Banach-Hausdorff-Lebesgue theorem \cite[Corollary~9.6]{Bra05} applied
to the corresponding realizers on Baire space.
\end{proof}

The last mentioned item can also be expressed such that $\lim^{[n]}$
is {\em effectively $\SO{n+1}$--complete} in the Weihrauch lattice. 

An important problem in the Weihrauch lattice is {\em closed choice} ${\C_X: \In\AA_-(X)\mto X}$, ${A\mapsto A}$,
which maps every closed set $A\In X$ to its points. 
By $\AA_-(X)$ we denote the set of closed subsets of a (computable) metric space $X$ with respect to negative information.
This means that a closed set is essentially described by enumerating basic open balls that exhaust
its complement. 
Intuitively, closed choice $\C_X$ is the following problem: given a closed set $A$ by a description that lists everything that does not belong to $A$, find a point $x\in A$ (see \cite{BBP12} for further information and precise definitions).

By $\BWT_X:\In X^\IN\mto X$ we denote the {\em Bolzano-Weierstra\ss{} theorem} for the
space $X$, which is the problem: given a sequence $(x_n)_n$ with a compact closure, find a cluster point of this sequence.
This problem has been introduced and studied in detail in \cite{BGM12}.
If $X$ itself is compact, then this problem is total. We are particularly interested in finite 
versions of the Bolzano-Weierstra\ss{} theorem here. 
We recall that we identify $k\in\IN$ with the set $X=\{0,...,k-1\}$ in the following.
By $\WKL$ we denote {\em weak K\H{o}nig's lemma}, which is the problem 
$\WKL:\In \Tr\to2^\IN,T\mapsto[T]$ that maps every infinite binary tree $T$ to
the set of its infinite paths. 
We summarize a number of facts that are of particular importance for us.
Here and in the following we will occasionally use that jumps are monotone with respect to strong 
Weihrauch reducibility, i.e., $f\leqSW g\TO f'\leqSW g'$ \cite[Proposition~5.6(2)]{BGM12}.

\begin{fact}
\label{fact:WKL-BWT}
We obtain
\begin{enumerate}
\item $\BWT_k^{(n-1)}\equivSW\C_k^{(n)}$ for all $k\in\IN$, $n\geq1$,
\item $\WKL^{(n)}\equivSW\widehat{\C_k^{(n)}}$ for all $k\geq2$ and $n\in\IN$,
\item $\lim^{(n)}\equivSW\widehat{\lim_k^{(n)}}$ for all $k\geq 2$, $k=\IN$ and $n\in\IN$,
\item $\lim_k^{(n)}\lSW\BWT_k^{(n)}$ and $\lim_k^{(n)}\lW\BWT_k^{(n)}$ for all $k\geq2$, $k=\IN$ and $n\in\IN$,
\item $\WKL^{(n)}\lSW\lim^{(n)}\lSW\WKL^{(n+1)}$ and $\WKL^{(n)}\lW\lim^{(n)}\lW\WKL^{(n+1)}$ for all $n\in\IN$,
\item $(\WKL')^{[n]}\equivW\WKL^{(n)}$ and $\lim^{[n]}\equivW\lim^{(n-1)}$ for all $n\geq1$.
\end{enumerate}
\end{fact}
\begin{proof}
(1) follows from \cite[Theorem~9.4]{BGM12} (noting that the cluster point problem used there is identical to the
Bolzano-Weierstra\ss{} theorem in the finite case).
(2) follows since $\WKL\equivSW\widehat{\C_k}$, which has essentially been proved in \cite[Theorem~8.2]{BG11},
where $\LLPO$ is just a reformulation of $\C_2$ and from the fact that jumps and parallelizations commute by
\cite[Proposition~5.7(3)]{BGM12}.
For (3) see \cite[Fact~3.5]{BGM12}.
(4) The reductions follow since the limit can be seen as a special case of the Bolzano-Weiertra\ss{} theorem \cite[Proposition~11.21]{BGM12} and they are strict since $\lim_k^{(n)}$ is $\SO{n+2}$--measurable,
but $\BWT_k^{(n)}$ for $k\geq2$ is not, see the proof of \cite[Proposition~9.1]{BGH15a}.
The positive parts of the reductions in (5) are easy to see, and the strictness follows from (relativized versions) of the low basis theorem and (relativized versions) of the Kleene tree construction, respectively (see also \cite{BBP12,BGM12}).
(6) The statement for $\WKL$ follows by induction from $\WKL^{(n+1)}\equivW\WKL^{(n)}*\WKL'$, where ``$\leqW$'' holds since
$\lim\leqW\WKL'$ and the converse reduction follows from \cite[Theorem~8.13]{BGM12}.  
The statement for $\lim$ follows for instance from \cite[Corollary~5.16]{BGM12} since $\lim^{(n)}$ is a cylinder.
\end{proof}

\section{The Uniform Scenario, Lower Bounds and Products}
\label{sec:uniform}

In this section we plan to introduce the uniform versions of Ramsey's theorem that we are going to study,
and we will prove some basic facts about them.
While in non-uniform settings such as reverse mathematics \cite{Sim09} certain information on infinite homogeneous sets
can just be assumed to be available, we need to make this more explicit.
For instance, it turned out to be useful to consider an enriched version $\CRT_{n,k}$ of Ramsey's theorem that
provides additional information on the color of the produced infinite homogeneous set. 
In \cite{CJS01} the stable version of Ramsey's theorem $\SRT_{n,k}$ was introduced, which 
is a restriction of Ramsey's theorem that we consider too.

We need to introduce some notation first.
For every $n\geq1$ we assume that we use some total standard numbering $\vartheta_n:\IN\to[\IN]^n$ that can be defined, for instance, 
by
$\vartheta_n\langle i_0,i_1,...,i_{n-1}\rangle:=
  \left\{k+\sum_{j=0}^{k}i_j:k=0,...,n-1\right\}$
for all $i_0,...,i_{n-1}\in\IN$, i.e., the set $\vartheta_n\langle i_0,i_1,...,i_{n-1}\rangle$ contains the numbers 
$i_0<i_0+i_1+1<i_0+i_1+i_2+2<...<i_0+i_1+i_{n-1}+n-1$. 
However, we will not make any technical use of this specific definition of $\vartheta_n$.
Occasionally, we use the notation $\{i_0<i_1<...<i_{n-1}\}$ for a set $\{i_0,i_1,...,i_{n-1}\}\in[\IN]^n$ with the additional property that
$i_0<i_1<...<i_{n-1}$. 
By $\CC_{n,k}$ we denote the set of colorings $c:[\IN]^n\to k$ which is represented 
such that $p\in\IN^\IN$ is a name for $c$, if $p(i)=c(\vartheta_n(i))$ (this is equivalent to using the
natural function space representation $[\vartheta_n\to\id_k]$ as known in computable analysis \cite{BHW08,Wei00}).
By $\CC_{n,\IN}$ we denote the set of all colorings $c:[\IN]^n\to\IN$ with a finite range, represented analogously. 
By $\HH_c$ we denote the set of infinite homogeneous sets $M\In\IN$ for the coloring $c$.
A coloring $c:[\IN]^n\to k$ is called {\em stable}, if $\lim_{i\to\infty}c(A\cup\{i\})$ exists for all $A\in[\IN]^{n-1}$.
The expression $k\geq a,\IN$ with $a\in\IN$ is means $k\in\{x\in\IN:x\geq a\}\cup\{\IN\}$.

\begin{definition}[Uniform variants of Ramsey's theorem]
For all $n\geq 1$ and $k\geq 1,\IN$ we define
\begin{enumerate}
\item $\RT_{n,k}:\CC_{n,k}\mto2^\IN,\RT_{n,k}(c):=\HH_c$,
\item $\CRT_{n,k}:\CC_{n,k}\mto k\times2^\IN,\CRT_{n,k}(c):=\{(c(M),M):M\in\HH_c\}$,
\item $\SRT_{n,k}:\In\CC_{n,k}\mto2^\IN,\SRT_{n,k}(c):=\RT_{n,k}(c)$,\\ where $\dom(\SRT_{n,k}):=\{c\in\CC_{n,k}:c$ stable$\}$,
\item $\CSRT_{n,k}:\In\CC_{n,k}\mto k\times 2^\IN,\CSRT_{n,k}(c):=\{(c(M),M):M\in\HH_c\}$,\\ where $\dom(\CSRT_{n,k}):=\{c\in\CC_{n,k}:c$ stable$\}$,
\item $\RT_{n,+}:=\bigsqcup_{k\geq1}\RT_{n,k}$, $\RT:=\bigsqcup_{n\geq1}\RT_{n,+}$.
\end{enumerate}
\end{definition}

All formalized versions of Ramsey's theorem mentioned here are well-defined by Ramsey's theorem~\ref{thm:Ramsey}.
We call $n$ the {\em cardinality} of the respective version of Ramsey's theorem.
Here $\CRT_{n,k}$ enriches $\RT_{n,k}$ by the information on the color of the resulting infinite homogeneous set, and
$\SRT_{n,k}$ is a restriction of $\RT_{n,k}$ defined only for stable colorings. 
The coproduct $\RT_{n,+}$ as well as $\RT_{n,\IN}$ can both be seen as different uniform versions of the principle $(\forall k)\;\RT_{n,k}$ that is usually
denoted by $\RT^n_{<\infty}$ in reverse mathematics (see, e.g., \cite{Hir15}).
In the case of $\RT_{n,\IN}$ the finite number of colors is left unspecified, whereas in the case of $\RT_{n,+}$, the number
of colors is an input parameter.

We emphasize that we use Cantor space $2^\IN$, i.e., the infinite homogeneous sets $M\in\HH_c$ are represented via their
characteristic functions $\chi_M:\IN\to\{0,1\}$. However, by definition any infinite subset $A\In M$ of an infinite homogeneous set $M\in\HH_c$
is in $\HH_c$ too, and given an enumeration of an infinite set $M$, we can find a characteristic function $\chi_A$ of an infinite subset $A\In M$.
This means that we can equivalently think about sets in $2^\IN$ as being represented via enumerations.\footnote{More formally, we could equivalently consider
the output space of $\RT_{n,k}$ and its variants as $\AA_+(\IN)$ or $\AA(\IN)$, i.e., as space of subsets of $\IN$ equipped with
positive or full information, respectively, which corresponds topologically to the lower Fell topology and the Fell topology, respectively.}

We obtain the following obvious strong reductions between the different versions of Ramsey's theorem.

\begin{lemma}[Basic reductions]
\label{lem:basic-reductions}
$\SRT_{n,k}\leqSW\RT_{n,k}\leqSW\CRT_{n,k}$ and \linebreak
$\SRT_{n,k}\leqSW\CSRT_{n,k}\leqSW\CRT_{n,k}$ for all $n\geq1$ and $k\geq1,\IN$.
\end{lemma}

We note that the colored versions of Ramsey's theorem are Weihrauch equivalent to the corresponding uncolored versions.
This is because given a coloring $c$ and an infinite homogeneous set $M\in\HH_c$,
we can easily compute $c(M)$ by choosing some points $i_0<i_1<...<i_{n-1}$ in $M$ and by computing $c\{i_0,i_1,...,i_{n-1}\}$.
Hence, we obtain the following corollary.

\begin{corollary}
\label{cor:basic-equivalences}
$\RT_{n,k}\equivW\CRT_{n,k}$ and $\SRT_{n,k}\equivW\CSRT_{n,k}$ for all $n\geq1$, $k\geq1,\IN$.
\end{corollary}

\begin{figure}[htb]
\begin{tikzpicture}[scale=.5,auto=left,every node/.style={fill=black!15}]

\def\rvdots{\raisebox{1mm}[\height][\depth]{$\huge\vdots$}};

  % Nodes
   \node (lim) at (8,17) {$\lim^{(n)}\equivSW\widehat{\lim_2^{(n)}}$};
  \node (WKL) at (8,14.5) {$\WKL^{(n)}\equivSW\widehat{\C_2^{(n)}}$};
%  \node (idRT) at (8,13) {$\id\times\RT_{n,k}$};
  \node (CRT) at (8,12) {$\CRT_{n,k}$};
  \node (RT) at (8,10.5) {$\RT_{n,k}$};
  \node (CSRT) at (8,7.5) {$\CSRT_{n,k}$};
  \node (SRT) at (8,6) {$\SRT_{n,k}$};
  \node (C2) at (8,3.5) {$\BWT_2^{(n-1)}\equivSW\C_2^{(n)}$};
  \node (lim2) at (8,1.5) {$\lim_2^{(n-1)}$};

  % Straight lines
  \foreach \from/\to in {
  lim/WKL,
  WKL/CRT,
  CRT/RT,
  CSRT/SRT,
  C2/lim2}
  \draw [->,thick] (\from) -- (\to);

  % Dashed lines
\draw[->,thick,dashed] (RT) to (CSRT);
\draw[->,thick,dashed] (SRT) to (C2);

  % Special curved lines
 \draw [->,thick,looseness=1.5] (CRT) to [out=190,in=170] (CSRT);
  \draw [->,thick,looseness=1.5] (RT) to [out=350,in=10] (SRT);
   \draw [->,thick,looseness=1.5] (CSRT) to [out=190,in=160] (C2);

\draw  (8,11.25) ellipse (2.2 and 2);
\draw  (8,6.75) ellipse (2.2 and 2);
\draw  (1,15.75) rectangle (15,2.5);

\end{tikzpicture}
  
\ \\[-0.5cm]
\caption{The degree of Ramsey's theorem  for fixed $n\geq2$ and $k\geq2,\IN$ in the Weihrauch lattice: all solid arrows indicate strong Weihrauch reductions against the direction of the arrow, all dashed arrows indicate ordinary Weihrauch reductions; the circle indicates what falls into a single Weihrauch degree, and the square box indicates what falls into a single parallelized Weihrauch degree.}
\label{fig:diagram-RT}
\end{figure}
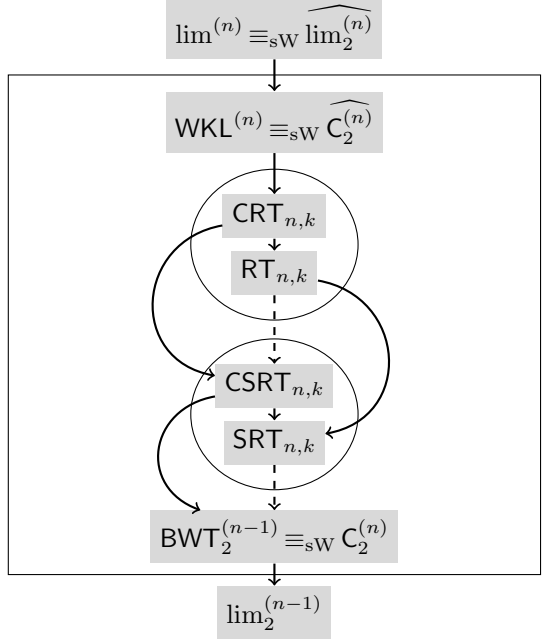

The diagram in Figure~\ref{fig:diagram-RT} illustrates the situation, and it displays further information on lower and upper bounds that is justified by proofs that we will only provide
step by step.  The diagram illustrates the non-trivial case $n\geq2$, $k\geq2,\IN$ whereas the bottom cases where $n=1$ or $k=1$ are described in the following result.

\begin{proposition}[The bottom cases]
\label{prop:bottom}
Let $n\geq1$ and $k\geq 2,\IN$. Then we obtain 
\begin{enumerate}
\item $\CSRT_{n,1}\equivSW\SRT_{n,1}\equivSW\RT_{n,1}\equivSW\CRT_{n,1}$  are computable and not cylinders,
\item $\lim_k\equivW\SRT_{1,k}\lW\BWT_k\equivW\RT_{1,k}$,
\item $\lim_k\leqSW\CSRT_{1,k}$ and $\BWT_k\leqSW\CRT_{1,k}$.
\end{enumerate}
\end{proposition}
\begin{proof}
We note that $\CSRT_{n,1},\SRT_{n,1},\RT_{n,1}$ and $\CRT_{n,1}$ can yield any infinite subset of $\IN$ (together with the color $1$ in the case of $\CSRT_{n,1},\CRT_{n,1}$), 
and hence they are all constant as multi-valued functions, hence computable and strongly equivalent to each other. Since the identity cannot be strongly reduced to constant multi-valued problems, 
it follows that all the aforementioned problems are not cylinders.
It is easy to see that $\lim_k\equivW\SRT_{1,k}$ and $\BWT_k\equivW\RT_{1,k}$, whereas $\lim_k\lW\BWT_k$ is given by Fact~\ref{fact:WKL-BWT}(4).
The strong reductions $\lim_k\leqSW\CSRT_{1,k}$ and $\BWT_k\leqSW\CRT_{1,k}$ are also clear.
\end{proof}

This result demonstrates that the Ramsey theorem for cardinality $1$ and $k$ colors $\RT_{1,k}$ is equivalent to the Bolzano-Weierstra\ss{} $\BWT_k$ for the space $k$, 
and hence the general Ramsey theorem can be seen as a generalization of the Bolzano-Weierstra\ss{} theorem for finite spaces. 
Next we want to prove the following interesting lower bound on the complexity of Ramsey's theorem.

\begin{theorem}[Lower bound]
\label{thm:lower-bound}
$\BWT_2^{(n-1)}\equivSW\C_2^{(n)}\leqSW\CSRT_{n,2}$ for all $n\geq2$.
\end{theorem}
\begin{proof}
We note that by Fact~\ref{fact:WKL-BWT} $\C_2^{(n)}\equivSW\BWT_2^{(n-1)}\equivSW\BWT_2\circ\lim_{2^\IN}^{[n-1]}$.
Let $p\in\dom(\BWT_2\circ\lim_{2^\IN}^{[n-1]})$ and $q:=\lim_{2^\IN}^{[n-1]}(p)$.
Then
\[q(i_0)=\lim_{i_1\to\infty}\lim_{i_2\to\infty}...\lim_{i_{n-1}\to\infty}p\langle i_{n-1},...,i_0\rangle\]
for all $i_0\in\IN$.
We compute the coloring $c:[\IN]^n\to 2$ with
\[c\{i_0<i_1<....<i_{n-1}\}:=p\langle i_{n-1},i_{n-2},...,i_1,i_0\rangle.\]
It is clear that $c$ is a stable coloring, and with the help of $\CSRT_{n,2}$ we can compute
$(c(M),M)\in\CSRT_{n,2}(c)$.
We claim that $c(M)\in\BWT_2(q)$.
This proves $\BWT_2^{(n-1)}\leqSW\CSRT_{n,2}$.
We still need to prove the claim.\\
\underline{1.\ Case:} $q$ is a convergent binary sequence, i.e., $x:=\lim_2(q)\in\{0,1\}$ exists.\\
In this case $q(i_0)=x$ for all sufficiently large $i_0$ and $\BWT_2(q)=\{x\}$.
Since $M$ is infinite, there will be such a sufficiently large $i_0\in M$.
Since $q=\lim_{2^\IN}^{[n-1]}(p)$ it follows that there will be sufficiently
large $i_{n-1}>...>i_1>i_0$ in $M$ such that  $p\langle i_{n-1},i_{n-2},...,i_1,i_0\rangle=x$,
and hence $c(M)=x\in\{x\}=\BWT_2(q)$.\\
\underline{2.\ Case:} $q$ is not a convergent binary sequence.\\
In this case $\BWT_{2}(q)=\{0,1\}$ and $c(M)\in\BWT_{2}(q)$ hold automatically. 
\end{proof}

It is clear that Ramsey's theorem for any number of colors $k$ can be reduced to Ramsey's theorem for any greater
number of colors. This is because any coloring $c:[\IN]^n\to k$ can be seen as a coloring for any number $m\geq k$ of colors, and
this idea applies to all versions of Ramsey's theorem that we have considered.

\begin{lemma}[Increasing colors]
\label{lem:increasing-color}
For all $n,k\geq1$ we obtain:
\begin{enumerate}
\item $\SRT_{n,k}\leqSW\SRT_{n,k+1}\leqSW\SRT_{n,\IN}$, 
\item $\CSRT_{n,k}\leqSW\CSRT_{n,k+1}\leqSW\CSRT_{n,\IN}$, 
\item $\RT_{n,k}\leqSW\RT_{n,k+1}\leqSW\RT_{n,+}\leqSW\RT_{n,\IN}$, 
\item $\CRT_{n,k}\leqSW\CRT_{n,k+1}\leqSW\CRT_{n,\IN}$.
\end{enumerate}
\end{lemma}

We will later come back to the question whether these reductions are strict.
In particular, with Theorem~\ref{thm:lower-bound} and Lemma~\ref{lem:increasing-color} 
we have now established the lower bound that is indicated in the diagram in Figure~\ref{fig:diagram-RT}.

\begin{corollary}[Lower bound]
\label{cor:lower-bound}
$\BWT_2^{(n-1)}\leqSW\CSRT_{n,k}$ for all $n\geq2$, $k\geq2,\IN$ and 
$\lim\nolimits_2^{(n-1)}\leqSW\CSRT_{n,k}$ for all $n\geq1$ and $k\geq2,\IN$.
\end{corollary}

The second statement follows from the first one by Fact~\ref{fact:WKL-BWT} in the case of $n\geq2$ and follows from Proposition~\ref{prop:bottom} in the case of $n=1$.
We note that by Corollary~\ref{cor:LLPO-RT} below $\CSRT_{n,k}$ cannot be replaced by $\SRT_{n,k}$ in this result.
It follows from Theorem~\ref{thm:CJS-lower} that the binary limit $\lim\nolimits_2$ cannot
be replaced by the limit on Baire space in the previous result.
We recall that $\lim^{(n-1)}\equivSW\lim^{[n]}$.

\begin{corollary}[Limit avoidance]
\label{cor:limit-avoidance1}
$\lim^{(n-1)}\nleqW\lim*\RT_{n,\IN}$ for all $n\geq2$.
\end{corollary}
\begin{proof}
Let us assume that $\lim^{(n-1)}\leqW\lim*\RT_{n,\IN}$ holds for some $n\geq2$.
Then $\lim^{(n-1)}$ maps some computable $p\in \IN^\IN$ to $\emptyset^{(n)}$, and the reduction maps $p$ 
to a computable coloring $c:[\IN]^{n}\to\IN$ with finite $\range(c)$.
By Theorem~\ref{thm:CJS-lower} there exists an infinite homogeneous set $M\In\IN$ for $c$ such that $\emptyset^{(n)}\nleqT M'$.
Now the assumption is that there is a limit computation performed on $p$ and $M$ that produces $\emptyset^{(n)}$.
But any result produced by such a limit computation can also be computed from $M'$ since $p$ is computable (and for 
all computable functions $F,G:\In\IN\to\IN$ there is a computable function $H:\In\IN\to\IN$ such that $F\circ\lim\circ\;G=H\circ\J$).
Hence, $\emptyset^{(n)}\leqT M'$, which is a contradiction.
\end{proof}

We obtain the following corollary since $\lim*\lim^{(n-2)}\equivW\lim^{(n-1)}$, which in turn follows from
\cite[Corollary~5.17]{BGM12} since $*$ is associative.

\begin{corollary}[Limit avoidance]
\label{cor:limit-avoidance}
$\lim^{(n-2)}\nleqW\RT_{n,\IN}$ for all $n\geq2$.
\end{corollary}

We recall that a problem $f$ is called {\em parallelizable}, if $f\equivW\widehat{f}$.
We can conclude  that Ramsey's theorem is not parallelizable. 

\begin{corollary}[Parallelizability]
\label{cor:parallelizability}
$\RT_{n,k}$ and $\SRT_{n,k}$ are not parallelizable for all $n\geq1$ and $k\geq2,\IN$.
\end{corollary}
\begin{proof}
By  Corollaries~\ref{cor:lower-bound} and \ref{cor:basic-equivalences} we have $\lim_2^{(n-1)}\leqW\SRT_{n,k}$ for $n\geq 1$, $k\geq2,\IN$.
Let us assume that $\SRT_{n,k}$ is parallelizable. 
Since parallelization is a closure operator, this implies by Lemma~\ref{lem:increasing-color}
\[\lim\nolimits^{(n-2)}\leqW\lim\nolimits^{(n-1)}\equivW\widehat{\lim\nolimits_2^{(n-1)}}\leqW\widehat{\SRT_{n,k}}\equivW\SRT_{n,k}\leqW\RT_{n,\IN},\] 
where  the first equivalence holds by Fact~\ref{fact:WKL-BWT}~(3).
This is a contradiction to Corollary~\ref{cor:limit-avoidance}.
The proof for $\RT_{n,k}$ follows analogously. 
\end{proof}

We recall that a problem $f$ is called {\em idempotent}, if $f\equivW f\times f$.
The squashing theorem (Theorem~\ref{thm:squashing}) implies that every total, finitely tolerant and idempotent problem is parallelizable.
Hence we can draw the following conclusion from Corollary~\ref{cor:parallelizability} (this has also been stated in \cite[Lemma~3.3]{DDH+16}).

\begin{corollary}[Idempotency]
\label{cor:no-idempotency}
$\RT_{n,k}$ is not idempotent for all $n\geq1$ and $k\geq2$.
\end{corollary}

Another noticeable consequence of Corollary~\ref{cor:lower-bound} is the following.

\begin{corollary}[Cylinders]
\label{cor:cylinder-CSRT}
$\widehat{\CSRT_{n,k}^{(m)}}$ and $\widehat{\CRT_{n,k}^{(m)}}$ are cylinders for all $n\geq1$, $k\geq2,\IN$ and $m\geq0$.
\end{corollary}

The claim follows since 
$\id\leqSW\lim\equivSW\widehat{\lim\nolimits_2}\leqSW\widehat{\CSRT_{n,k}}$
holds by Corollary~\ref{cor:lower-bound} and Fact~\ref{fact:WKL-BWT}.
The next lemma now formulates a simple finiteness condition that Ramsey's theorem satisfies
and implies that Ramsey's theorem has very little uniform computational power.

\begin{lemma}[Finite Intersection Lemma]
\label{lem:finite-intersection}
Let $c_i:[\IN]^n\to k$ be colorings for $i=1,...,m$ with $m,n\geq1$, $k\geq1,\IN$. Then we obtain $\bigcap_{i=1}^m\RT_{n,k}(c_i)\not=\emptyset$.
\end{lemma}
\begin{proof}
We first consider $k\geq1$.
We use a bijection $\alpha:\{0,1,...,k-1\}^m\to\{0,1,...,k^m-1\}$ in order to construct a map $f:(\CC_{n,k})^m\to\CC_{n,k^m}$
by
\[f(c_1,...,c_m)(A):=\alpha(c_1(A),...,c_m(A))\]
for all colorings $c_1,...,c_m\in\CC_{n,k}$ and $A\in[\IN]^n$.
Given $c_1,...,c_m\in\CC_{n,k}$ we consider $c:=f(c_1,...,c_m)$, and we claim that
$\RT_{n,k^m}(c)\In\bigcap_{i=1}^m\RT_{n,k}(c_i)$. 
To this end, let $M\in\RT_{n,k^m}(c)$ and $x:=c(M)$. 
If $(x_1,...,x_m):=\alpha^{-1}(x)$, then we obtain $c_i(A)=x_i$ for all $i=1,...,m$ and $A\in[M]^n$,
and hence $M$ is homogeneous for all $c_1,...,c_m$. This proves the claim.
It follows by Ramsey's theorem~\ref{thm:Ramsey} that $\RT_{n,k^m}(c)\not=\emptyset$.

We now consider the case $k=\IN$. In this case we use Cantor's tupling function $\alpha:\IN^m\to\IN$ in order to construct
a map $f:(\CC_{n,\IN})^m\to\CC_{n,\IN}$ analogously as above. We obtain $\RT_{n,\IN}(c)\In\bigcap_{i=1}^m\RT_{n,\IN}(c_i)$.
It follows by Ramsey's theorem~\ref{thm:Ramsey} that $\RT_{n,\IN}(c)\not=\emptyset$.

In any case this proves the claim of the lemma.
\end{proof}

This result could be generalized and proved in different possibly simpler ways. 
The specific construction used in our proof will be reused for the proof of Corollary~\ref{cor:products}.

We recall that for a multi-valued function $f:\In X\mto Y$ we have defined $\# f$, the {\em cardinality} of $f$, in \cite{BGH15a} as the supremum of all cardinalities of sets $M\In\dom(f)$
such that $\{f(x):x\in M\}$ is pairwise disjoint. The cardinality yields an invariant for strong Weihrauch reducibility, i.e.,
$f\leqSW g$ implies $\# f\leq \#g$ \cite[Proposition~3.6]{BGH15a}.
Lemma~\ref{lem:finite-intersection} now implies the following perhaps surprising fact.

\begin{corollary}[Cardinality]
\label{cor:cardinality-RT}
$\#\RT_{n,k}^{(m)}=\#\widehat{\RT_{n,k}^{(m)}}=1$ for all $n\geq1$, $k\geq1,\IN$ and $m\geq0$.
\end{corollary}

In the case of the parallelization we only need to apply Lemma~\ref{lem:finite-intersection} pairwise to any component of the sequence.
We recall that a multi-valued function $f$ was called {\em discriminative} in \cite{BHK15} if $\C_2\leqW f$ and {\em indiscriminative} otherwise.
Likewise, we could call $f$ {\em strongly discriminative} if $\C_2\leqSW f$. We obtain $\#\C_2=2$ since $\{0\}$ and $\{1\}$ are in the domain of $\C_2$. Hence it follows that Corollary~\ref{cor:cardinality-RT}
implies in particular that Ramsey's theorem is not strongly discriminative, not even in its parallelized form.

\begin{corollary}[Strong discrimination]
\label{cor:LLPO-RT}
$\C_2\nleqSW\widehat{\RT_{n,k}^{(m)}}$ for all $n\geq1$, $k\geq1,\IN$ and $m\geq 0$.
\end{corollary}

Since $\C_2'\equivSW\BWT_2$ by Fact~\ref{fact:WKL-BWT}, we can conclude that $\CSRT$ cannot be replaced by $\SRT$
in Theorem~\ref{thm:lower-bound} and Corollary~\ref{cor:lower-bound} (without simultaneously replacing the strong reduction by an ordinary one).\footnote{It follows from Corollary~\ref{cor:DNC} below that in Corollary~\ref{cor:LLPO-RT} $\C_2$ cannot be replaced by $\ACC_\IN$, as defined in \cite{BHK15}.}
We can also conclude from Corollary~\ref{cor:cardinality-RT}
that the parallelized uncolored versions of Ramsey's theorem are not cylinders since $\#\id=|\IN^\IN|$.

\begin{corollary}[Cylinders]
\label{cor:cylinder-SRT}
$\widehat{\RT_{n,k}^{(m)}}$ and $\widehat{\SRT_{n,k}^{(m)}}$ are not cylinders for all $n\geq1$, $k\geq1,\IN$ and $m\geq 0$.
\end{corollary}

Since $\#\CSRT_{n,k}\geq k$ holds (because there are monochromatic colorings for each color), we can also conclude from Corollary~\ref{cor:cardinality-RT} that the colored versions
of Ramsey's theorem are not strongly equivalent to the uncolored ones (in the case of at least two colors).

\begin{corollary}
\label{cor:CSRT-RT}
$\CSRT_{n,k}\nleqSW\widehat{\RT_{n,k}^{(m)}}$ for all $n\geq1$, $k\geq2,\IN$ and $m\geq0$.
\end{corollary}

The following result is a consequence of Lemma~\ref{lem:finite-intersection} and its proof.
For a finite sequence $(f_i)_{i\leq m}$ of multi-valued functions $f_i:\In X\mto Y$ we denote the {\em intersection} by
$\bigcap_{i=1}^mf_i:\In X\mto Y,x\mapsto\bigcap_{i=1}^mf_i(x)$, where
$\dom(\bigcap_{i=1}^mf_i)$ contains all points $x\in X$ such that $\bigcap_{i=1}^mf_i(x)\not=\emptyset$.
We recall that $f^m$ denotes the $m$--fold product of $f$ with respect to $\times$.

\begin{corollary}[Products]
\label{cor:products}
For all $n,m,k\geq1$ we obtain
\begin{enumerate}
\item $\RT_{n,k}^m\leqSW\bigcap_{i=1}^m\RT_{n,k}\leqSW\RT_{n,k^m}$,
\item $\SRT_{n,k}^m\leqSW\bigcap_{i=1}^m\SRT_{n,k}\leqSW\SRT_{n,k^m}$,
\item $\RT_{n,\IN}^m\leqSW\bigcap_{i=1}^m\RT_{n,\IN}\leqSW\RT_{n,\IN}$,
\item $\SRT_{n,\IN}^m\leqSW\bigcap_{i=1}^m\SRT_{n,\IN}\leqSW\SRT_{n,\IN}$.
\end{enumerate}
\end{corollary}
\begin{proof}
The functions $f:(\CC_{n,k})^m\to\CC_{n,k^m}$ and $f:(\CC_{n,\IN})^m\to\CC_{n,\IN}$ constructed in the proof of 
Lemma~\ref{lem:finite-intersection} are computable, and hence they yield the 
reductions $\bigcap_{i=1}^m\RT_{n,k}\leqSW\RT_{n,k^m}$ and $\bigcap_{i=1}^m\RT_{n,\IN}\leqSW\RT_{n,\IN}$, respectively.
Additionally, both maps $f$ have the property that they map stable colorings $c_1,...,c_m$ to
stable colorings $c:=f(c_1,...,c_m)$, hence they also yield the corresponding reductions in the stable case.
The other reductions follow from Lemma~\ref{lem:finite-intersection}.
\end{proof}

We note that by \cite[Proposition~2.1]{DDH+16} we also have $\RT_{n,k}\times\RT_{n,l}\leqSW\RT_{n,kl}$ for all $n,k,l\geq1$. We also note that the following result is implicitly included in \cite[Section~1]{DDH+16} (the proof given
there is for the case $l=1$ and $k=2$ and can be generalized straightforwardly).

\begin{proposition}[Compositional products]
\label{prop:compositional-products}
$\RT_{n,k+l}\leqW \RT_{n,k}*\RT_{n,l+1}$ for all $n,k,l\geq1$.
\end{proposition}

From Corollary~\ref{cor:products} we can directly conclude that Ramsey's theorem for an unspecified finite
number of colors is idempotent.

\begin{corollary}[Idempotency]
\label{cor:idempotency}
$\RT_{n,\IN}$ and $\SRT_{n,\IN}$ are idempotent for all $n\geq1$.
\end{corollary}

Since the sequence $(f_m)_m$ of maps $f_m:(\CC_{n,\IN})^m\to\CC_{n,\IN}$ from 
the proof of Corollary~\ref{cor:products} is uniformly computable, we obtain the following corollary. 

\begin{corollary}[Finite parallelization]
\label{cor:finite-parallelization}
$\RT_{n,k}^*\leqW\RT_{n,+}$ and 
$\SRT_{n,k}^*\leqW\SRT_{n,+}$ for all $n\geq1$ and $k\geq1,\IN$.
\end{corollary}

We note that $\RT_{n,k}^*=\bigsqcup_{m\geq0}\RT_{n,k}^m$, where $\RT_{n,k}^0=\id$, and hence
we obtain only an ordinary Weihrauch reduction in the previous result. 
 Corollary~\ref{cor:finite-parallelization} leads to the obvious question, whether additional factors can make up for color increases.

\begin{question}[Colors and factors]
\label{quest:colors-factors}
Does $\RT_{n,k}^*\equivW\RT_{n,+}$ hold for $n,k\geq	2$?
\end{question}

We note that the equivalence is known to hold in the case of $n=1$. 

\begin{proposition}[Colors and factors]
\label{prop:colors-factors}
$\RT_{1,n+1}\leqW\RT_{1,2}^n$ for all $n\geq1$ and, in particular, $\RT_{1,2}^*\equivW\RT_{1,+}$.
\end{proposition}
\begin{proof}
In \cite[Theorem~32]{Pau10} it was proved that $\C_{n+1}\leqSW\C_2^n$ holds for all $n\geq1$ 
(only ``$\leqW$'' was claimed but the proof shows ``$\leqSW$'').
By Fact~\ref{fact:WKL-BWT}(1) this implies $\BWT_{n+1}\leqSW\BWT_2^n$, since jumps are monotone with respect to $\leqSW$. Hence, with Proposition~\ref{prop:bottom} we obtain
$\RT_{1,n+1}\equivW\BWT_{n+1}\leqW\BWT_2^n\equivW\RT_{1,2}^n$.
Since this reduction is uniform in $n$, we can conclude that 
\[\RT_{1,+}=\bigsqcup_{k\geq1}\RT_{1,k}\leqW\bigsqcup_{k\geq1}\RT_{1,2}^{k-1}\equivW\RT_{1,2}^*.\]
The other direction follows since the reductions in Corollary~\ref{cor:products} are uniform and hence
\[\RT_{1,2}^*=\bigsqcup_{m\in\IN}\RT_{1,2}^m\leqW\bigsqcup_{m\in\IN}\RT_{1,2^m}\leqW\bigsqcup_{k\geq1}\RT_{1,k}=\RT_{1,+}.\qedhere\]
\end{proof}

As a consequence of the next result we obtain that any unspecified finite number of colors can be reduced to two colors
for the price of an increase of the cardinality. Simultaneously, the following theorem also gives us a handle to show
that the complexity of Ramsey's theorem increases with increasing numbers of colors.

\begin{theorem}[Products]
\label{thm:products}
$\RT_{n,\IN}\times\RT_{n+1,k}\leqSW\RT_{n+1,k+1}$ and\linebreak 
$\SRT_{n,\IN}\times\SRT_{n+1,k}\leqSW\SRT_{n+1,k+1}$ for all $n\geq1$ and $k\geq1,\IN$.
\end{theorem}
\begin{proof}
We start with the first reduction.
Given a coloring $c_1:[\IN]^n\to\IN$ with finite range and a coloring $c_2:[\IN]^{n+1}\to k$ we construct a coloring
$c^+:[\IN]^{n+1}\to k+1$ as follows:
\[c^+(A):=\left\{\begin{array}{ll}
c_2(A) & \mbox{if $A$ is homogeneous for $c_1$}\\
k    & \mbox{otherwise}
\end{array}\right.\]
for all $A\in[\IN]^{n+1}$.
Let $M\in\RT_{n+1,k+1}(c^+)$ and let $p:\IN\to\IN$ be the principal function of $M$. 
Then we define a coloring $c_p:[\IN]^n\to\IN$ by $c_p(A):=c_1(p(A))$ for all
$A\in[\IN]^n$. By construction, $c_p$ has finite range too. Let $B\in\RT_{n,\IN}(c_p)$. Then $p(B)$ is homogeneous for $c_1$
and $p(B)\In M$. Hence any $A\in[p(B)]^{n+1}$ is also homogeneous for $c_1$,
which implies $c^+(A)=c_2(A)$ and hence $c^+(M)=c_2(A)<k$. This implies $M\in\RT_{n+1,k}(c_2)$
and all $A\in[M]^{n+1}$ are homogeneous for $c_1$. We claim that this implies $M\in\RT_{n,\IN}(c_1)$. 
The claim yields $\RT_{n+1,k+1}(c^+)\In\RT_{n,\IN}(c_1)\cap\RT_{n+1,k}(c_2)$, and hence the desired reduction follows.

We still need to prove the claim. To this end we show that all $A\in [M]^{n+1}$ are homogeneous
for $c_1$ with respect to one fixed color.  Firstly, we note that for every two $A,B\in[M]^{n+1}$ 
there is a finite sequence $A_1,...,A_l\in[M]^{n+1}$ such that $A_1=A$, $A_l=B$ and $|A_i\cap A_{i+1}|\geq n$
for all $i=1,...,l-1$. This is because each element of $A$ can be replaced step by step by one element of $B$.
Now we fix some $i\in\{1,...,l-1\}$.
Since $A_i$ and $A_{i+1}$ are homogeneous for $c_1$ by assumption and they share an $n$--element subset,
it is clear that $c_1(A_i)=c_1(A_{i+1})$. Since this holds for all $i\in\{1,...,l-1\}$, we obtain $c_1(A)=c_1(B)$.
This means that all $A\in[M]^{n+1}$ are homogeneous for $c_1$ with respect to the same fixed color,
and hence, in particular, all $A\in[M]^n$ share the same color with respect to $c_1$, i.e., $M\in\RT_{n,\IN}(c_1)$.

Now we still show that the same construction also proves the second reduction regarding stable colorings.
For this it suffices to show that $c^+:[\IN]^{n+1}\to k+1$ is stable for all stable colorings $c_1:[\IN]^n\to\IN$ and $c_2:[\IN]^{n+1}\to k$.
For this purpose let $c_1,c_2$ be stable, and let $A\in[\IN]^n$. 
Then $[A]^{n-1}=\{B_0,B_1,...,B_{n-1}\}$ and for each $i=0,...,n-1$
there is some $l_i\geq\max(B_i)$ and some $x_i\in\IN$ such that $c_1(B_i\cup\{j\})=x_i$ for $j>l_i$, since $c_1$ is stable.
There is also some $l_n\geq\max(A)$ and some $x_n\in k$ such that $c_2(A\cup\{j\})=x_n$ for $j>l_n$, since $c_2$ is stable.
Let $l:=\max\{l_0,...,l_{n-1},l_n\}$. Then 
\[c^+(A\cup\{j\})=\left\{\begin{array}{ll}
c_2(A\cup\{j\}) & \mbox{if $c_1(A)=x_0=x_1=...=x_{n-1}$}\\
k & \mbox{otherwise}
\end{array}\right.\]
for all $j>l$. Hence $c^+$ is stable. 
\end{proof}

We note that in the case of $k=1$ we get the following corollary. 

\begin{corollary}[Color reduction]
\label{cor:color-reduction}
$\RT_{n,\IN}\leqSW\RT_{n+1,2}$ and $\SRT_{n,\IN}\leqSW\SRT_{n+1,2}$ for all $n\geq1$.
\end{corollary}

We note that the increase of the cardinality from $n$ to $n+1$ in this corollary is necessary in the case of $\RT$ by Corollaries~\ref{cor:no-idempotency} and \ref{cor:finite-parallelization}.

We mention that the proof of Theorem~\ref{thm:products} also shows that the coloring $c^+$ constructed therein has only
infinite homogeneous sets of colors other than $k$. 
In the case of $k=1$ this means that only infinite homogeneous sets of color $0$ occur.
We denote by $0\dash\SRT_{n+1,2}$ the stable version $\SRT_{n+1,2}$
of Ramsey's theorem restricted to colorings that only admit infinite homogeneous sets of color $0$.
We obtain the following corollary, which we consider only as a technical step towards the
proof of Corollary~\ref{cor:color-reduction-jumps} in the next section (and hence we do not phrase it for $\RT$ in place of $\SRT$).

\begin{corollary}[Color reduction]
\label{cor:color-reduction-zero}
$\SRT_{n,\IN}\leqSW0\dash\SRT_{n+1,2}$ for all $n\geq1$.
\end{corollary}

With the help of Proposition~\ref{prop:bottom}(2) in the case $k=\IN$ we obtain the following corollary of Corollary~\ref{cor:color-reduction}.

\begin{corollary}[Discrete lower bounds]
\label{cor:discrete-lower-bounds}
$\lim_\IN\leqW\SRT_{2,2}$ and $\BWT_\IN\leqW\RT_{2,2}$.
\end{corollary}

Since Ramsey's theorem is not parallelizable by Corollary~\ref{cor:parallelizability}, it is clear that some increase in the 
cardinality is necessary in order to accommodate the parallelization. 
We will show that it is sufficient and necessary to increase the cardinality by $2$.
The next result shows that this is sufficient, and it uses the ideas that have already been applied in the proofs of 
Corollary~\ref{cor:products} and Theorem~\ref{thm:products}.

\begin{theorem}[Delayed Parallelization]
\label{thm:delayed-parallelization}
$\widehat{\RT_{n,k}}\leqSW\RT_{n+2,2}$ and\linebreak $\widehat{\SRT_{n,k}}\leqSW\SRT_{n+2,2}$ for all $n\geq1$ and $k\geq1,\IN$.
\end{theorem}
\begin{proof}
We start with the reduction $\widehat{\RT_{n,k}}\leqSW\RT_{n+2,2}$ for $k\geq1$.
Given a sequence $(c_i)_i$ of colorings $c_i:[\IN]^n\to k$, we want to determine
infinite homogeneous sets $M_i$ for all of them in parallel, using $\RT_{n+2,2}$.
The sequence $(f_m)_m$ of functions $f_m:(\CC_{n,k})^m\to\CC_{n,k^m}$, defined as in the proof of Lemma~\ref{lem:finite-intersection}, 
is computable, and we use it to compute a sequence $(d_m)_m$ of colorings $d_m\in\CC_{n,k^m}$ by
$d_m:=f_m(c_0,...,c_{m-1})$. Given the sequence $(d_m)_m$, we can compute a sequence $(d_m^+)_m$ 
of colorings $d_m^+:[\IN]^{n+1}\to2$ by
\[d_m^+(A):=\left\{\begin{array}{ll}
0 & \mbox{if $A$ is homogeneous for $d_m$}\\
1 & \mbox{otherwise}
\end{array}\right.\]
for all $A\in[\IN]^{n+1}$.
Now, in a final step we compute a coloring $c:[\IN]^{n+2}\to2$ with
\[c(\{m\}\cup A):=d_m^+(A)\]
for all $A\in[\IN]^{n+1}$ and $m<\min(A)$. Given an infinite homogeneous set $M\in\RT_{n+2,2}(c)$
we determine a sequence $(M_i)_i$ as follows: for each fixed $i\in\IN$ we first search for a number $m>i$ in $M$, and then we let $M_i:=\{x\in M:x>m\}$.
It follows from the definition of $c$ that $M_i$ is homogeneous for $d_m^+$, and 
following the reasoning in the proof of Theorem~\ref{thm:products}, we obtain that $M_i$ is also homogeneous for $d_m$.  
Following the reasoning in the proof of Lemma~\ref{lem:finite-intersection}, we finally conclude that $M_i\in\bigcap_{j=0}^{m-1}\RT_{n,k}(c_j)$,
hence, in particular, $M_i\in\RT_{n,k}(c_i)$, which was to be proved. 

We note that the entire construction preserves stability. As shown in the proof of Corollary~\ref{cor:products}, the function
$f$ preserves stability. Hence, given a sequence $(c_i)_i$ of stable colorings, also the sequence $(d_m)_m$ consists
of stable colorings. Likewise, it was shown in the proof of  Theorem~\ref{thm:products} that in this case
also the sequence $(d_m^+)_m$ consists of stable colorings.
It follows immediately from the construction of $c$ that also $c$ is stable, since
\[\lim_{j\to\infty}c(\{m\}\cup A\cup\{j\})=\lim_{j\to\infty}d_m(A\cup\{j\})\]
for all $A\in[\IN]^{n}$ and $m<\min(A)$. Altogether, this proves $\widehat{\SRT_{n,k}}\leqSW\SRT_{n+2,2}$.
The case $k=\IN$ can be handled analogously. 
\end{proof}

Again the observation made after Corollary~\ref{cor:color-reduction} applies: the colorings $d_m^+$ can only
have infinite homogeneous sets of color $0$, and hence also $c$ can only have infinite homogeneous sets of color $0$.
This yields the following corollary, which we consider as a technical step towards the proof of
Corollary~\ref{cor:lower bounds with jumps} in the next section (hence we do not phrase it for $\RT$ in place of $\SRT$).

\begin{corollary}[Delayed Parallelization]
\label{cor:delayed-parallelization-color}
$\widehat{\SRT_{n,k}}\leqSW0\dash\SRT_{n+2,2}$ for all $n\geq1$ and $k\geq1,\IN$.
\end{corollary}

Theorem~\ref{thm:delayed-parallelization} in combination with some other results also yields
the following lower bounds of versions of Ramsey's theorem.

\begin{corollary}[Lower bounds]
\label{cor:delayed-parallelization}
$\lim\leqW\SRT_{3,2}$, $\WKL'\leqW\RT_{3,2}$ and\linebreak $\WKL^{(n)}\leqW\SRT_{n+2,2}$ for all $n\geq2$.
\end{corollary}
\begin{proof}
With the help of Fact~\ref{fact:WKL-BWT}, Corollary~\ref{cor:basic-equivalences} and Theorem~\ref{thm:delayed-parallelization} we obtain:
\begin{itemize}
\item $\lim\equivW\widehat{\lim_2}\equivW\widehat{\SRT_{1,2}}\leqW\SRT_{3,2}$,
\item $\WKL'\equivW\widehat{\C_2'}\leqW\widehat{\RT_{1,2}}\leqW\RT_{3,2}$,
\item $\WKL^{(n)}\equivW\widehat{\C_2^{(n)}}\leqW\widehat{\SRT_{n,2}}\leqW\SRT_{n+2,2}$ for $n\geq2$.
\end{itemize}
For the first statement we have additionally used Proposition~\ref{prop:bottom} and in the latter two cases
Theorem~\ref{thm:lower-bound}.
\end{proof}

Corollary~\ref{cor:delayed-parallelization} generalizes \cite[Corollary~2.3]{HJ16} (see Corollary~\ref{cor:KL}).
Corollary~\ref{cor:limit-avoidance} together with Corollary~\ref{cor:delayed-parallelization}
show that both corollaries are optimal in the sense that 
$\lim^{(n-2)}$ cannot be replaced by $\lim^{(n-3)}$ in the statement of the Corollary~\ref{cor:limit-avoidance},
and $\WKL^{(n)}$ cannot be replaced by $\lim^{(n)}$ in the statement of Corollary~\ref{cor:delayed-parallelization}.
In particular, $n+2$ in Theorem~\ref{thm:delayed-parallelization} is also optimal and cannot be replaced by $n+1$.
However, the following question remains open.

\begin{question}
\label{quest:WKL-SRT32}
Does $\WKL'\leqW\SRT_{3,2}$ hold?
\end{question}

As a combination of Corollaries~\ref{cor:discrete-lower-bounds} and \ref{cor:delayed-parallelization}
we obtain the following result on versions of Ramsey's theorem above the limit and the Bolzano-Weierstra\ss{} theorem
for $\{0,1\}, \IN$ and $\IN^\IN$, respectively. 

\begin{corollary}[Cones]
\label{cor:cones}
$\lim_2\leqSW\BWT_2\leqW\RT_{1,2}$, $\lim_\IN\leqSW\BWT_\IN\leqW\RT_{2,2}$ and $\lim\leqSW\BWT_{\IN^\IN}\leqW\RT_{3,2}$.
\end{corollary}
\begin{proof}
We note that $\lim_X\leqW\BWT_X$ holds for arbitrary computable metric spaces $X$ by \cite[Proposition~11.21]{BGM12}.
Hence the first reduction chain follows with Proposition~\ref{prop:bottom}, the second one
follows from Corollary~\ref{cor:discrete-lower-bounds} and the third one from Corollary~\ref{cor:delayed-parallelization}
with the extra observation that $\WKL'\equivW\BWT_{\IN^\IN}$  by \cite[Corollaries 11.6 and 11.7]{BGM12}.
\end{proof}

\section{Jumps, Increasing Cardinality and Color and Upper Bounds}
\label{sec:upper}

The purpose of this section is to provide a useful upper bound on Ramsey's theorem.
Simultaneously, we will demonstrate how the complexity of Ramsey's theorem increases with increasing cardinality.
The proof is subdivided into several steps.  
The first and crucial step made by Theorem~\ref{thm:CRT-SRT} is interesting by itself and connects the jump of the colored version of Ramsey's theorem
with the stable version of the next greater cardinality. This is one result where the usage of the colored
version of Ramsey's theorem is essential. We start with a result that prepares the first direction of this theorem.

\begin{proposition}[Jumps]
\label{prop:jumps}
$\CRT_{n,k}'\leqSW\CSRT_{n+1,k}$  for all $n\geq1$ and $k\geq1,\IN$.
\end{proposition}
\begin{proof}
Let $(c_i)_i$ be a sequence that converges to a coloring $c_\infty:[\IN]^n\to k$.
Without loss of generality we can assume that the $c_i$ are colorings $c_i:[\IN]^n\to k$ themselves
(this can easily be achieved by replacing every value $\geq k$ in the range of $c_i$ by $0$).
We compute the coloring $c:[\IN]^{n+1}\to k$ with
\[c(A\cup\{i\}):=c_i(A)\]
for all $A\in[\IN]^n$ and $i>\max(A)$.
Then $c$ is stable, and we claim that $\RT_{n+1,k}(c)\In\RT_{n,k}(c_\infty)$.
To this end, let $M\in\RT_{n+1,k}(c)$, and let $A\in[M]^n$. 
Since $M$ is infinite and $\lim_{i\to\infty}c(A\cup\{i\})=c_\infty(A)$, we obtain $c(M)=c_\infty(A)$.
Since this holds for all $A\in[M]^n$, we obtain that $M$ is homogeneous for $c_\infty$, i.e., $M\in\RT_{n,k}(c_\infty)$.
We note that we also obtain $c_\infty(M)=c(M)$, and hence this proves $\CRT_{n,k}'\leqSW\CSRT_{n+1,k}$.
\end{proof}

We obtain the following corollary, since the jump operator is monotone with respect to strong reductions
by \cite[Proposition~5.6]{BGM12}.

\begin{corollary}[Jumps]
\label{cor:jumps}
$\RT_{n,k}^{(m)}\leqW\SRT_{{n+m},k}$ for all $m,n\geq1$, $k\geq1,\IN$.
\end{corollary}

It is not immediately clear whether this result also holds with a strong reduction.
Now we are prepared to formulate our main result on jumps of Ramsey's theorem.

\begin{theorem}[Jumps]
\label{thm:CRT-SRT}
$\CRT_{n,k}'\equivW\SRT_{n+1,k}$ for all $n\geq1$, $k\geq1,\IN$.
\end{theorem}
\begin{proof}
By Proposition~\ref{prop:jumps} and Corollary~\ref{cor:basic-equivalences} we obtain $\CRT_{n,k}'\leqW\SRT_{n+1,k}$, 
and it remains to prove $\SRT_{n+1,k}\leqW\CRT_{n,k}'$.
Let $c:[\IN]^{n+1}\to k$ be a stable coloring. Then we can define a 
sequence $(c_i)_i$ of colorings $c_i:[\IN]^n\to k$ by
\[c_i(A):=\left\{\begin{array}{ll}
  c(A\cup\{i\}) & \mbox{if $i>\max(A)$}\\
  0                  & \mbox{otherwise}
\end{array}\right.\]
for all $A\in[\IN]^n$.
Then $(c_i)_i$ is a converging sequence of colorings, and let the limit be denoted by $c_\infty:[\IN]^n\to k$. 
With the help of $\CRT_{n,k}$ we can compute $(c_\infty(M_\infty),M_\infty)\in\CRT_{n,k}(c_\infty)$.
We will now describe how we can use this set $M_\infty$ together with $c_\infty$ and $c$ in order
to computably enumerate an infinite homogeneous set $M\in\SRT_{n+1,k}(c)$.
The set $M=\bigcup_{i=0}^\infty M_i$ will be defined inductively using sets $M_i\in[M_\infty]^{n+i}$.
We start with choosing some $M_0\in[M_\infty]^n$. Then we continue in steps $i=0,1,2...$ as follows:

Let us assume that we have $M_i\in [M_\infty]^{n+i}$. For all $A\in[M_i]^{n}$ we obtain
\[\lim_{j\to\infty} c(A\cup\{j\})=c_\infty(A)=c_\infty(M_\infty).\]
Hence, we can effectively find an $m>\max(M_i)$ such that $M_{i+1}:=M_i\cup\{m\}$ satisfies
$c(M_{i+1})=c_\infty(M_\infty)$. 

Since all the sets $M_i$ with $i\geq1$ are homogeneous sets for $c$ with the same color $c_\infty(M_\infty)$, the set $M:=\bigcup_{i=0}^\infty M_i$
is an infinite homogeneous set for $c$ with $c(M)=c_\infty(M_\infty)$. This proves the desired reduction $\SRT_{n+1,k}\leqW\CRT_{n,k}'$.
\end{proof}

It follows from Corollaries~\ref{cor:cylinder-CSRT} and \ref{cor:cylinder-SRT} 
that the equivalence in Theorem~\ref{thm:CRT-SRT} cannot be replaced
by a strong equivalence.
We note that the color provided by $\CRT_{n,k}'$ is not needed if the color of the infinite homogeneous set that is to be constructed is known
in advance. Hence, the proof of Theorem~\ref{thm:CRT-SRT} yields also the following result.

\begin{corollary}[Jumps in the case of known color]
\label{cor:CRT-SRT}
$0\dash\SRT_{n+1,2}\leqW\RT_{n,2}'$ for all $n\geq1$.
\end{corollary}

We note that Corollaries~\ref{cor:CRT-SRT} and \ref{cor:color-reduction-zero} allow us to improve the bound given in Corollary~\ref{cor:color-reduction}
somewhat.

\begin{corollary}[Color reduction with jumps]
\label{cor:color-reduction-jumps}
$\SRT_{n,\IN}\leqSW\RT_{n,2}'$ for all $n\geq1$.
\end{corollary}

We collect all a number of lower bound results in the following corollary
that strengthen some earlier results
mentioned in Corollaries~\ref{cor:discrete-lower-bounds} and \ref{cor:delayed-parallelization}.

\begin{corollary}[Lower bounds with jumps]
\label{cor:lower bounds with jumps}
$\lim_\IN\leqW\RT_{1,2}'$, $\lim\leqW\RT_{2,2}'$ and\linebreak 
$\WKL^{(n)}\leqW\RT_{n+1,2}'$ for all $n\geq2$.
\end{corollary}
\begin{proof}
We obtain $\lim_\IN\equivW\SRT_{1,\IN}\leqW\RT_{1,2}'$ by Proposition~\ref{prop:bottom} and Corollary~\ref{cor:color-reduction-jumps}.
With the help of Fact~\ref{fact:WKL-BWT}, Proposition~\ref{prop:bottom}, Theorem~\ref{thm:lower-bound} and Corollaries~\ref{cor:delayed-parallelization-color} and \ref{cor:CRT-SRT} we obtain
\begin{itemize}
\item $\lim\equivW\widehat{\lim_2}\equivW\widehat{\SRT_{1,2}}\leqSW 0\dash\SRT_{3,2}\leqW\RT_{2,2}'$,
\item $\WKL^{(n)}\equivW\widehat{\C_2^{(n)}}\leqW\widehat{\SRT_{n,2}}\leqW0\dash\SRT_{n+2,2}\leqW\RT_{n+1,2}'$
for $n\geq2$.\qedhere
\end{itemize}
\end{proof}

Analogously to Question~\ref{quest:WKL-SRT32} the following question remains open.

\begin{question}
\label{quest:WKL-RT22}
Does $\WKL'\leqW\RT_{2,2}'$ hold?
\end{question}

Roughly speaking, Theorem~\ref{thm:CRT-SRT} indicates that any increase in the cardinality of Ramsey's theorem corresponds
to a jump. 
We can also conclude from Corollary~\ref{cor:jumps} that Ramsey's theorem is increasing with respect
to increasing cardinality. 

\begin{lemma}[Increasing cardinality]
\label{lem:increasing-size}
$\SRT_{n,k}\leqW\RT_{n,k}\leqW\SRT_{n+1,k}\leqW \RT_{n+1,k}$ for all $n\geq1$, $k\geq1,\IN$.
\end{lemma}
\begin{proof}
It follows from Lemma~\ref{lem:basic-reductions} and Corollaries~\ref{cor:jumps} and \ref{cor:basic-equivalences} that\linebreak
$\SRT_{n,k}\leqW\RT_{n,k}\leqW\RT_{n,k}'\leqSW\CRT_{n,k}'\leqW\SRT_{n+1,k}\leqW\RT_{n+1,k}$.
\end{proof}

We will soon see in Corollary~\ref{cor:increasing-size} and \ref{cor:RT22-SRT2N} that the reductions in this lemma are all strict in certain cases.
We note that $\CRT_{n,k}'\leqW\CRT_{n,k}*\lim\equivW\RT_{n,k}*\lim$. The latter problem is very stable and has several useful descriptions.

\begin{lemma}[Jump of the cylindrification]
\label{lem:jump-cylindrification}
For all $n\geq1$, $k\geq1,\IN$ we obtain
$\RT_{n,k}*\lim\equivW(\RT_{n,k}\times\id)'\equivW\RT_{n,k}'\times\lim$.
\end{lemma}
\begin{proof}
The second equivalence holds since jumps and products commute by \cite[Proposition~5.7~(2)]{BGM12}.
Since $\RT_{n,k}\times\id$ is a cylinder we have  by \cite[Corollary~5.17]{BGM12}
\[(\RT_{n,k}\times\id)'\equivW(\RT_{n,k}\times\id)*\lim\equivW(\RT_{n,k}*\lim)\times\lim\equivW\RT_{n,k}*\lim.\]
The latter two equivalences hold since $\lim$ is idempotent.
\end{proof}

Since $\RT_{n,k}\times\id$ is the cylindrification of Ramsey's theorem, this lemma characterizes the
jump of the cylindrification of Ramsey's theorem up to Weihrauch equivalence.
In particular, we obtain the following consequence of Theorem~\ref{thm:CRT-SRT}. 

\begin{corollary}
\label{cor:SRT-RT-lim}
$\SRT_{n+1,k}\leqW\RT_{n,k}*\lim$ for all $n\geq1$, $k\geq1,\IN$.
\end{corollary}

This result is the crucial step towards our upper bound result. 
The next step involves a usage of the cohesiveness problem.
We recall that a set $A\In\IN$ is called {\em cohesive} for a sequence $(R_i)_i$ of sets $R_i\In\IN$ if
$A\cap R_i$ or $A\cap(\IN\setminus R_i)$ is finite for each $i\in\IN$. In other words, up to finitely many
exceptions, $A$ is fully included in any $R_i$ or its complement. By $\COH:(2^\IN)^\IN\mto2^\IN$ we denote
the cohesiveness problem, where $\COH(R_i)$ contains all sets $A\In\IN$ that are cohesive for $(R_i)_i$.
The uniform computational content of $\COH$ has already been studied in \cite{DDH+16,BHK15,Dzh15}.
The relevance of the cohesiveness problem for Ramsey's theorem has originally been noticed by
Cholak et al.\ who proved over $\RCA_0$ that $\RT_{2,2}\iff\SRT_{2,2}\wedge\COH$ \cite[Lemma~7.11]{CJS01} (see \cite{CJS09} for a correction). 
We use the same idea to prove the following, which was in the case of $n=k=2$ also observed by Dorais et al.\ \cite{DDH+16}.

\begin{proposition}
\label{prop:RT-SRT-COH}
$\RT_{n,k}\leqW\SRT_{n,k}*\COH$ for all $n\geq1$, $k\geq1,\IN$.
\end{proposition}
\begin{proof}
We fix some $n,k\geq1$.
Given a coloring $c:[\IN]^{n}\to k$ we compute a sequence $(R_i)_i$ of sets $R_i\In \IN$ as follows:
\[R_{\langle i,j\rangle}:=\{r\geq\max\vartheta_{n-1}(i):c(\vartheta_{n-1}(i)\cup\{r\})=j\}\]
for all $i,j\in\IN$. Here $\vartheta_{n-1}$ denotes the numbering of $[\IN]^{n-1}$ introduced in Section~\ref{sec:uniform}.
With the help of $\COH$ we can compute
an infinite cohesive set $Y\in\COH(R_i)_i$ for the sequence $(R_i)_i$.
Let $\sigma:\IN\to\IN$ be the principal function of $Y$.
Since $Y$ is cohesive for $(R_i)_i$ and $\range(c)$ is finite, it follows that for all $B\in[\IN]^{n-1}$ there is some $j\in\IN$ such that
$c(B\cup\{r\})=j$
holds for almost all $r\in Y$.
Hence, the coloring $c_\sigma:[\IN]^n\to k$, defined by 
\[c_\sigma(A):=c(\sigma(A))\]
for all $A\in[\IN]^n$ is stable.
With the help of $\SRT_{n,k}$ we can compute an infinite homogeneous set $M_\sigma\in\SRT_{n,k}(c_\sigma)$.
It is clear that $M:=\sigma(M_\sigma)$ is an infinite homogeneous set for $c$.
\end{proof}

Now we can combine Corollary~\ref{cor:SRT-RT-lim} with Proposition~\ref{prop:RT-SRT-COH}
in both possible orders in order to obtain the following result.
We recall that the compositional product $*$ is associative \cite[Proposition~31]{BP16}.

\begin{corollary}
\label{cor:RT-SRT-lim-COH}
For all $n\geq1$, $k\geq1,\IN$ we obtain
\begin{enumerate}
\item $\RT_{n+1,k}\leqW\RT_{n,k}*\lim*\COH$,
\item $\SRT_{n+1,k}\leqW\SRT_{n,k}*\COH*\lim$.
\end{enumerate}
\end{corollary}

The diagram in Figure~\ref{fig:diagram-RT-COH-lim} illustrates the situation.
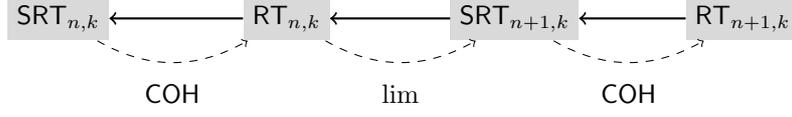
\begin{figure}[htb]
\begin{tikzpicture}[scale=.5,auto=left,every node/.style={fill=black!15}]

\def\rvdots{\raisebox{1mm}[\height][\depth]{$\huge\vdots$}};

  % Nodes

%  \node (idRT) at (8,13) {$\id\times\RT_{n,k}$};

  \node (SRT) at (0,3) {$\SRT_{n,k}$};
  \node (RT) at (6,3) {$\RT_{n,k}$};
  \node (SRT1) at (12,3) {$\SRT_{n+1,k}$};
  \node (RT1) at (18,3) {$\RT_{n+1,k}$};

  % Straight lines
  \foreach \from/\to in {
RT1/SRT1,
SRT1/RT,
RT/SRT}
  \draw [->,thick] (\from) -- (\to);

  % Dashed lines
%\draw[->,thick,dashed] (RT) to (CSRT);
%\draw[->,thick,dashed] (SRT) to (C2);

  % Special curved lines
 \draw [->,dashed,looseness=1] (SRT) to [out=330,in=210] (RT);
  \draw [->,dashed,looseness=1] (RT) to [out=330,in=210] (SRT1);
  \draw [->,dashed,looseness=1] (SRT1) to [out=330,in=210] (RT1);

\node [fill=none] at (3,1) {$\COH$};
\node [fill=none] at (9,1) {$\lim$};
\node [fill=none] at (15,1) {$\COH$};

\end{tikzpicture}
  
\ \\[-0.5cm]
\caption{Upper bounds for Ramsey's theorem for $n\geq2$, $k\geq2,\IN$: all solid arrows indicated strong Weihrauch reductions against the direction of the arrow, all circled dashed arrows are labeled with an upper bound of the inverse implication.}
\label{fig:diagram-RT-COH-lim}
\end{figure}
The first bound given by Corollary~\ref{cor:RT-SRT-lim-COH} is particularly useful, since the following result was proved in \cite[Corollary~14.14]{BHK15}.

\begin{fact}
\label{fact:WKL-lim-COH}
$\WKL'\equivW\lim*\COH$.
\end{fact}

Fact~\ref{fact:WKL-lim-COH} together with Corollary~\ref{cor:RT-SRT-lim-COH} yields the following.

\begin{corollary}[Induction]
\label{cor:induction}
$\RT_{n+1,k}\leqW\RT_{n,k}*\WKL'$ for all $n\geq1$, $k\geq1,\IN$.
\end{corollary}

This corollary can be interpreted such that $\WKL'$ is sufficient to transfer $\RT_{n,k}$ into $\RT_{n+1,k}$. 
An obvious question is whether $\WKL'$ is minimal with this property.
It follows from Proposition~\ref{prop:RT2*RT1xlim} below and Lemma~\ref{lem:jump-cylindrification}
that $\WKL'$ cannot be replaced by $\lim$ in Corollary~\ref{cor:induction}.

Corollary~\ref{cor:induction} could also be proved directly using so-called Erd\H{o}s-Rado trees.\footnote{This approach has been used in \cite[Proposition~6.6.1]{Rak15};
for a proof theoretic analysis of Ramsey's theorem this method has been applied in \cite{KK09a,KK12,BS14}, and in reverse mathematics
it has been used to prove that Ramsey's theorem is provable over $\ACA_0$; see \cite{EHMR84a} for the Erd\H{o}s-Rado method in general.}
It reflects the fact that Ramsey's theorem can be proved inductively, where the induction base follows from the Bolzano-Weierstra\ss{} theorem for the
$k$--point space by
$\RT_{1,k}\equivW\BWT_k\leqW\WKL'$, which holds by Proposition~\ref{prop:bottom} and Fact~\ref{fact:WKL-BWT}.
Altogether, this means that we can derive the following result from Corollary~\ref{cor:induction} and Fact~\ref{fact:WKL-BWT}(6).

\begin{corollary}[Upper bound]
\label{cor:upper-bound}
$\RT_{n,k}\leqSW\WKL^{(n)}$ for all $n\geq1$, $k\geq1,\IN$.
\end{corollary}

We note that we get a strong reduction here, since $\WKL^{(n)}$ is a cylinder.
The upper bound is tight in terms of the number of jumps and also up to parallelization. 
Using the upper bound from Corollary~\ref{cor:upper-bound} together with the lower bound from Corollary~\ref{cor:delayed-parallelization}
yields the following characterization of Ramsey's theorem $\RT$.

\begin{corollary}
\label{cor:RT}
$\RT\equivW\bigsqcup_{n=0}^\infty\lim^{(n)}\equivW\bigsqcup_{n=0}^\infty\WKL^{(n)}$.
\end{corollary}

This degree corresponds to the class $\ACA_0'$ in reverse mathematics (see \cite[Theorem~6.27]{Hir15}).
We call a problem $f$  {\em countably irreducible} if $f\leqW\bigsqcup_{n=0}^\infty g_n$ implies
$f\leqW g_n$ for some $n\in\IN$.\footnote{This property has been called ``join-irreducible'' in previous publications, but formally it is a strengthening of join-irreducibility in the lattice theoretic sense. It is also not identical to ``countable join-irreducibility'' since the countable coproduct is not a supremum in general.}
Since it is clear that $\C_{\IN^\IN}$ is countably irreducible \cite[Corollary~5.6]{BBP12}, and it is easy to see that $\bigsqcup_{n=0}^\infty\lim^{(n)}\leqW\C_{\IN^\IN}$,
 we obtain the following corollary.

\begin{corollary}
\label{cor:RT-CNN}
$\RT\lW\C_{\IN^\IN}$.
\end{corollary}

Here $\C_{\IN^\IN}$ can be seen as a possible counterpart of $\ATR_0$ in reverse mathematics, in the sense
that some statements that are equivalent to $\ATR_0$ over $\RCA_0$ in reverse mathematics turn out to be equivalent to $\C_{\IN^\IN}$
in the Weihrauch lattice if interpreted as problems.\footnote{Results in this direction are not yet published, but this
emerged during a discussion at a recent Dagstuhl seminar on the subject.}
We obtain the following characterization of the parallelization of Ramsey's theorem.

\begin{corollary}[Parallelization]
\label{cor:parallelization}
$\widehat{\RT_{n,k}}\equivW\WKL^{(n)}\equivSW\widehat{\CRT_{n,k}}$ for all $n\geq1$, $k\geq2,\IN$ and
$\widehat{\SRT_{n,k}}\equivW\WKL^{(n)}\equivSW\widehat{\CSRT_{n,k}}$ for all $n\geq2$, $k\geq2,\IN$.
\end{corollary}
\begin{proof}
Using Fact~\ref{fact:WKL-BWT}, Corollary~\ref{cor:basic-equivalences}, the lower bound from Corollary~\ref{cor:lower-bound}, the upper bound from Corollary~\ref{cor:upper-bound} and the fact that parallelization is a closure operator
we obtain the following reduction chain for $n\geq2$, $k\geq2,\IN$:
\[\WKL^{(n)}\equivSW\widehat{\C_2^{(n)}}\equivSW\widehat{\BWT_2^{(n-1)}}\leqSW\widehat{\CSRT_{n,k}}\leqW\widehat{\SRT_{n,k}}
\leqSW\widehat{\RT_{n,k}}\leqSW\WKL^{(n)},\]
which implies $\widehat{\SRT_{n,k}}\equivW\widehat{\RT_{n,k}}\equivW\WKL^{(n)}$, and
the corresponding strong equivalences $\widehat{\CSRT_{n,k}}\equivSW\widehat{\CRT_{n,k}}\equivSW\WKL^{(n)}$
follow, since the colored versions of Ramsey's theorem are cylinders by Corollary~\ref{cor:cylinder-CSRT},
and $\WKL^{(n)}$ is a cylinder too.
The proof for $n=1$ follows similarly using Proposition~\ref{prop:bottom}. 
\end{proof}

We mention that $\widehat{\SRT_{1,k}}\equivW\lim\equivSW\widehat{\CSRT_{1,k}}$ holds for $k\geq2,\IN$ by Corollary~\ref{cor:lower-bound}.
By Corollary~\ref{cor:cylinder-SRT} the parallelized uncolored versions of Ramsey's theorem are not cylinders,
and hence the Weihrauch equivalences $\equivW$ in the previous corollary cannot be replaced by strong ones $\equivSW$.

Since $\WKL^{(n)}\leqW\lim^{(n)}$ and $\WKL^{(n)}\nleqW\lim^{(n-1)}$ for all $n\geq1$ by Fact~\ref{fact:WKL-BWT}
and $\lim^{(n)}$ is effectively $\SO{n+2}$--complete by Facts~\ref{fact:WKL-BWT} and \ref{fact:Borel}, we obtain the following corollary
that characterizes the Borel complexity of Ramsey's theorem. 

\begin{corollary}[Borel complexity]
\label{cor:Borel-measurability}
$\SRT_{n,k}$ and $\RT_{n,k}$ are both effectively $\SO{n+2}$--measurable, but not effectively $\SO{n+1}$--measurable for all $n\geq2$, $k\geq2,\IN$.
\end{corollary}

Both positive statements and the negative statement on $\RT_{n,k}$ also hold for $n=1$.
By an application of Facts~\ref{fact:Borel} and \ref{fact:WKL-BWT} we can also rephrase the negative statements as follows.

\begin{corollary}
\label{cor:SRT-limits}
$\SRT_{n,k}\nleqW\lim^{(n-1)}$ for $n\geq2$
and $\RT_{n,k}\nleqW\lim^{(n-1)}$ for $n\geq1$ and both for $k\geq2,\IN$.
\end{corollary}

Since effective $\SO{n}$--measurability is preserved downwards by Weihrauch reducibility by Fact~\ref{fact:Borel}, this
in turn implies that Ramsey's theorem actually forms a strictly increasing chain with increasing cardinality.

\begin{corollary}[Increasing cardinality]
\label{cor:increasing-size}
$\RT_{n,k}\lW\RT_{n+1,2}$, $\RT_{n,k}\lSW\RT_{n+1,2}$,\linebreak 
$\SRT_{n,k}\lW\SRT_{n+1,2}$ and $\SRT_{n,k}\lSW\SRT_{n+1,2}$ for all $n\geq1$ and $k\geq2,\IN$.
\end{corollary}

Here the positive parts of the reduction hold by Corollary~\ref{cor:color-reduction}.
The separations (at least in the case of $\RT$) do also follow from classical non-uniform 
separation results of Jockusch \cite{Joc72}
who showed that every computable instance of $\RT_{n,k}$ has a $\emptyset^{(n)}$--computable solution,
whereas there are computable instance of $\RT_{n+1,k}$ that have no $\emptyset^{(n)}$--computable solution.
We can also draw some conclusions on increasing numbers of colors.
In \cite[Theorem~3.1]{DDH+16} Dorais et al.\ have proved that $\RT_{n,k}\lSW\RT_{n,k+1}$ holds for all $n,k\geq1$.
Their main tool was a version of the squashing theorem (Theorem~\ref{thm:squashing}) for strong Weihrauch reducibility.
With the help of Theorem~\ref{thm:products} we can strengthen this result to ordinary Weihrauch reducibility, 
which answers \cite[Question~7.1]{DDH+16}. 
This result was independently obtained by Hirschfeldt and Jockusch~\cite[Theorem~3.3]{HJ16} and Patey~\cite[Corollary~3.15]{Pat16a}.

\begin{theorem}[Increasing numbers of colors]
\label{thm:increasing-colors}
$\RT_{n,k}\lW\RT_{n,k+1}$ for all $n,k\geq1$.
\end{theorem}
\begin{proof}
Let us assume that $\RT_{n,2}\times\RT_{n+1,k}\leqW\RT_{n+1,k}$ holds for some $n,k\geq1$.
Then by the squashing theorem (Theorem~\ref{thm:squashing}) we obtain $\widehat{\RT_{n,2}}\leqW\RT_{n+1,k}$, and hence
by Corollary~\ref{cor:parallelization}
\[\lim\nolimits^{(n-1)}\leqW\WKL^{(n)}\equivW\widehat{\RT_{n,2}}\leqW\RT_{n+1,k}\]
in contradiction to Corollary~\ref{cor:limit-avoidance}.
Hence $\RT_{n,2}\times\RT_{n+1,k}\nleqW\RT_{n+1,k}$ for all $n,k\geq1$.
On the other hand, we have $\RT_{n,2}\times\RT_{n+1,k}\leqW\RT_{n+1,k+1}$ by Theorem~\ref{thm:products}.
This implies $\RT_{n+1,k}\lW\RT_{n+1,k+1}$ for all $n,k\geq1$.
The claim $\RT_{1,k}\lW\RT_{1,k+1}$ for all $k\geq1$ was already known \cite[Theorem~13.4]{BGM12} via Proposition~\ref{prop:bottom}.
\end{proof}

From this result we can also conclude that the two uniform versions of $\RT^n_{<\infty}$ are not equivalent.
We recall that a problem $f:\In X\mto Y$ is called a {\em fractal}, if there is an $F:\In\IN^\IN\mto\IN^\IN$ with
$F\equivW f$ and $F|_A\equivW f$ for all clopen $A\In\IN^\IN$ such that $A\cap\dom(F)\not=\emptyset$.
Analogously, {\em strong fractals} are defined with $\equivSW$ in place of $\equivW$.
It is easy to see that $\RT_{n,\IN}$ and $\SRT_{n,\IN}$ are fractals
and hence countably irreducible by \cite[Proposition~2.6]{BGM12},
and thus the strictness of the reductions given in the following corollary follows from Theorem~\ref{thm:increasing-colors}.

\begin{corollary}[Arbitrary numbers of colors]
\label{cor:arbitrary-colors}
$\RT_{n,+}\lW\RT_{n,\IN}$ and \linebreak
$\SRT_{n,+}\lW\SRT_{n,\IN}$ for all $n\geq1$.
\end{corollary}

Another consequence of the squashing theorem (Theorem~\ref{thm:squashing}) is the following result that shows that
the complexity of Ramsey's theorem also grows with an increasing number of factors. 
This generalizes Corollary~\ref{cor:no-idempotency}.

\begin{proposition}[Increasing number of factors]
\label{prop:factors}
$\RT_{n,k}^m\lW\RT_{n,k}^{m+1}$ for all $n,m\geq1$ and $k\geq2$.
\end{proposition}
\begin{proof}
Let us assume that $\RT_{n,k}^{m+1}\equivW\RT_{n,k}\times\RT_{n,k}^m\leqW\RT_{n,k}^{m}$ holds for some $n,m\geq1$ and $k\geq2$.
Then by the squashing theorem (Theorem~\ref{thm:squashing}) we obtain $\widehat{\RT_{n,k}}\leqW\RT_{n,k}^m$, and hence by Corollaries~\ref{cor:finite-parallelization}, \ref{cor:color-reduction} and \ref{cor:parallelization} we can conclude
\[\lim\nolimits^{(n-1)}\leqW\WKL^{(n)}\equivW\widehat{\RT_{n,k}}\leqW\RT_{n,k}^m\leqW\RT_{n,k}^*\leqW\RT_{n+1,2},\]
in contradiction to Corollary~\ref{cor:limit-avoidance}.
Hence, $\RT_{n,k}^{m+1}\nleqW\RT_{n,k}^{m}$.
\end{proof}

We note that Question~\ref{quest:colors-factors} remains, whether additional factors can make up for color increases.
The characterization given in Corollary~\ref{cor:parallelization} also implies the following result on the arithmetic complexity of homogeneous sets
with the help of the uniform low basis theorem~\cite[Theorem~8.3]{BBP12}.

\begin{corollary}[Arithmetic complexity]
\label{cor:arithmetic-complexity}
Every computable sequence $(c_i)_i$ of colorings $c_i:[\IN]^n\to k$ for $n\geq1$ and $k\geq2$ admits a sequence $(M_i)_i$ 
such that $\langle M_0,M_1,...\rangle'\leqT\emptyset^{(n+1)}$ and
such that $M_i$ is an infinite homogeneous set for $c_i$ for each $i\in\IN$.
\end{corollary}
\begin{proof}
We use the map $\L:=\J^{-1}\circ\lim$, where $\J:\IN^\IN\to\IN^\IN,p\mapsto p'$ is the Turing jump operation.
The uniform low basis theorem~\cite[Theorem~8.3]{BBP12} states that $\WKL\leqSW\L$ holds,
which implies $\WKL^{(n)}\leqSW\L^{(n)}$, since jumps are monotone with respect to strong Weihrauch reducibility.
By Corollary~\ref{cor:parallelization} we have $\widehat{\RT_{n,k}}\equivW\WKL^{(n)}$, and hence any
instance $(c_i)_i$ of $\widehat{\RT_{n,k}}$ has a solution $M=\langle M_0,M_1,...\rangle$ whose jump $M'$ 
is computable in $\emptyset^{(n+1)}$.
\end{proof}

Here $n+1$ cannot be replaced by $n$ in $\emptyset^{(n+1)}$ by Corollary~\ref{cor:parallelization}, and in this sense this result is optimal.
Hence, we have a striking difference in the arithmetic complexity between the non-uniform Ramsey theorem as witnessed by Theorems~\ref{thm:Jockusch} and \ref{thm:CJS-upper} and
the uniform sequential version of Ramsey's theorem for sequences as witnessed by Corollary~\ref{cor:arithmetic-complexity}.
This yields another proof of Corollary~\ref{cor:parallelizability}.

Since jumps commute with parallelization by \cite[Proposition~5.7(3)]{BGM12}, Corollary~\ref{cor:parallelization} yields also the following corollary, which states that under parallelization
a jump of the colored versions of Ramsey's theorem corresponds exactly to an increase in cardinality by one.

\begin{corollary}[Parallelized jumps]
\label{cor:jumps-parallelization-CSRT}
For all $n\geq2$, $k\geq2,\IN$ we obtain\linebreak
$\widehat{\CSRT_{n,k}'}\equivSW\widehat{\CSRT_{n+1,k}}\equivSW\widehat{\CRT_{n,k}'}\equivSW\widehat{\CRT_{n+1,k}}$. 
\end{corollary}

A similar property also holds for the uncolored version of Ramsey's theorem $\RT_{n,k}$, but in this case it cannot be so easily concluded since
the parallelized uncolored theorem is not a cylinder by Corollary~\ref{cor:cylinder-SRT}.
Hence, we need the following reformulation of Corollary~\ref{cor:SRT-RT-lim} that is justified by Lemma~\ref{lem:jump-cylindrification}.

\begin{corollary}
\label{cor:SRT-RT-lim2}
$\SRT_{n+1,k}\leqW\RT_{n,k}'\times\lim$ for all $n\geq1$, $k\geq1,\IN$.
\end{corollary}

Now we are prepared to prove the following result on parallelized jumps of Ramsey's theorem.

\begin{corollary}[Parallelized jumps]
\label{cor:jumps-parallelization-SRT}
$\widehat{\RT_{n,k}'}\equivW\widehat{\RT_{n+1,k}}$ for all $n\geq1$ and $k\geq2,\IN$.
\end{corollary}
\begin{proof}
It follows from Corollary~\ref{cor:jumps} that $\RT_{n,k}'\leqW\RT_{n+1,k}$,
and hence one direction of the claims follows since parallelization is a closure operator.
By Corollaries~\ref{cor:parallelization}, \ref{cor:SRT-RT-lim2}, \ref{cor:lower bounds with jumps}, Fact~\ref{fact:WKL-BWT}~(3)
and since parallelization commutes with products it follows that
\[\widehat{\RT_{n+1,k}}\equivW\widehat{\SRT_{n+1,k}}\leqW\widehat{\RT_{n,k}'}\times\lim\equivW\widehat{\RT_{n,k}'}\times\widehat{\lim\nolimits_\IN}\leqW\widehat{\RT_{n,k}'},\]
which completes the proof.
\end{proof}

We leave it open whether a corresponding fact can be established for the stable version of Ramsey's theorem (which is not the case for $n=1$).
We have now completely provided all the positive information that is displayed in the diagram
in Figure~\ref{fig:diagram-RT}. What still remains to be done are some of the separations.

\section{Ramsey's Theorem for Pairs}
\label{sec:pairs}

In this section we want to discuss $\RT_{2,2}$, which is of particular interest.
For one, we derive some conclusions from the general results that we proved before, we mention some results
that have been proved by other authors, and we raise some open questions.
The neighborhood of Ramsey's theorem in the Weihrauch lattice is illustrated in the diagram in Figure~\ref{fig:diagram-RT22}.

\begin{figure}[htb]
\begin{tikzpicture}[scale=.5,auto=left,every node/.style={fill=black!15}]

\def\rvdots{\raisebox{1mm}[\height][\depth]{$\huge\vdots$}};

  % Nodes
   \node (limSS) at (2,18.75) {$\lim''$};
   \node (WKLSS) at (2,17.25) {$\WKL''$};
   \node (limS) at (2,10.75) {$\lim'$};
   \node (WKLS) at (2,9.25) {$\WKL'\equivSW\KL\equivSW\BWT_\IR$};
   \node (COH) at (9.25,6) {$\COH$};
   \node (BWT2) at (12.25,9.25) {$\C_2'\equivSW\BWT_2$};
   \node (BWT2S) at (12.25,12.5) {$\BWT_2'$};
   \node (lim) at (2,7.5) {$\lim$};
   \node (WKL) at (2,6) {$\WKL\equivSW\C_{2^\IN}$};
   \node (WWKL) at (2,4.5) {$\WWKL$};
   \node (PA) at (-1.25,6) {$\PA$};
   \node (MLR) at (-0.75,4.5) {$\MLR$};
   \node (DNC) at (2,3) {$\DNC_\IN$};
   \node (ACC) at (4.75,3) {$\ACC_\IN$};
   \node (LLPO) at (12.25,3) {$\C_2\equivSW\LLPO$};
   \node (LPO) at (12.25,4.5) {$\LPO$};
   \node (lim2) at (12.25,6) {$\lim_2$};
   \node (SRT12) at (18,6) {$\SRT_{1,2}$};
   \node (CN) at (15.25,7.5) {$\C_\IN$};
 %  \node (CF) at (20,8.5) {$\CF$};
  \node (CRT12S) at (18,14.25) {$\CRT_{1,2}'$};
  \node (RT22) at (14,15.75) {$\RT_{2,2}$};
  \node (SRT22) at (14,14.25) {$\SRT_{2,2}$};
  \node (SRT22COH) at (9.25,14.25) {$\SRT_{2,2}\sqcup\COH$};
  \node (SRT22xCOH) at (9.25,15.75) {$\SRT_{2,2}\times\COH$};
  \node (SRT22SCOH) at (9.25,17.25) {$\SRT_{2,2}*\COH$};
  \node (RT12) at (18,9.25) {$\RT_{1,2}=\D_{1,2}$};
  \node (RT12S) at (18,12.5) {$\RT_{1,2}'=\D_{2,2}$};
  \node (RT12lim) at (4.25,15.75) {$\RT_{1,2}'\times\lim$};
 
  % Straight lines
  \foreach \from/\to in {
  limSS/WKLSS,
  WKLSS/limS,
  WKLSS/SRT22SCOH,
   SRT22xCOH/SRT22COH,
  WKLSS/RT12lim,
  SRT22COH/COH,
  RT22/SRT22COH,
  SRT22COH/SRT22,
  CRT12S/BWT2S,
  CRT12S/RT12S,
  BWT2S/BWT2,
  RT12S/RT12,
  limS/WKLS,
  WKLS/lim,
  lim/COH,
  WKLS/BWT2,
  lim/WKL,
  lim/CN,
%  CN/CF,
  WKL/WWKL,
  WKL/PA,
  WWKL/MLR,
  WWKL/DNC,
  WWKL/LLPO,
  DNC/ACC,
  RT12/SRT12,
  BWT2/lim2,
  CN/lim2,
  lim2/LPO,
  LPO/LLPO,
  LLPO/ACC}
  \draw [->,thick] (\from) -- (\to);

  % Dashed lines
\draw[->,thick,dashed] (SRT22SCOH) to (RT22);
\draw[->,thick,dashed] (RT12lim) to (SRT22xCOH);
\draw[->,thick,dashed] (SRT22SCOH) to (SRT22xCOH);
%\draw[->,thick,dashed] (SRT) to (C2);

  %Curved lines
% \draw [->,thick,looseness=0.4] (RT12S) to [out=340,in=60] (CF);
 \draw [->,thick,looseness=2] (RT12S) to [out=190,in=30] (DNC);
  \draw [->,thick,looseness=2.5] (RT12lim) to [out=280,in=10] (lim);
  % \draw [->,thick,looseness=2] (RT) to [out=350,in=10] (SRT);
  % \draw [->,thick,looseness=1.5] (CSRT) to [out=190,in=170] (C2);

%Curved dashed lines
 \draw [->,thick,dashed,looseness=0.5] (BWT2) to [out=350,in=190] (RT12);
 \draw [->,thick,dashed,looseness=0.5] (RT12) to [out=170,in=10] (BWT2);
 \draw [->,thick,dashed,looseness=0.5] (SRT22) to [out=350,in=190] (CRT12S);
 \draw [->,thick,dashed,looseness=0.5] (CRT12S) to [out=170,in=10] (SRT22);
 \draw [->,thick,dashed,looseness=0.5] (lim2) to [out=350,in=190] (SRT12);
 \draw [->,thick,dashed,looseness=0.5] (SRT12) to [out=170,in=10] (lim2);
 \draw [->,thick,dashed,looseness=1.9] (RT12S) to [out=330,in=350] (CN);
% \draw [->,thick,dashed,looseness=0.5] (CF) to [out=160,in=10] (CN);

 % \draw [->,thick,looseness=2] (RT) to [out=350,in=10] (SRT);
  % \draw [->,thick,looseness=1.5] (CSRT) to [out=190,in=170] (C2);

\draw  (-2,8.25) rectangle (20.75,2.25);
\draw  (-2.25,11.5) rectangle (21,2);
\draw  (-2.5,19.5) rectangle (21.25,1.75);
\node  [fill=none] at (-1.75,18.75) {$\SO{4}$};
\node  [fill=none] at (-1.5,10.75) {$\SO{3}$};
\node  [fill=none] at (-1.25,7.5) {$\SO{2}$};

\end{tikzpicture}
  
\ \\[-0.5cm]
\caption{Ramsey's theorem  for pairs and two colors in the Weihrauch lattice: all solid arrows indicate strong Weihrauch reductions against the direction of the arrow, all dashed arrows indicate ordinary Weihrauch reductions, and
the boxes indicate the given levels of the effective Borel hierarchy.}
\label{fig:diagram-RT22}
\end{figure}
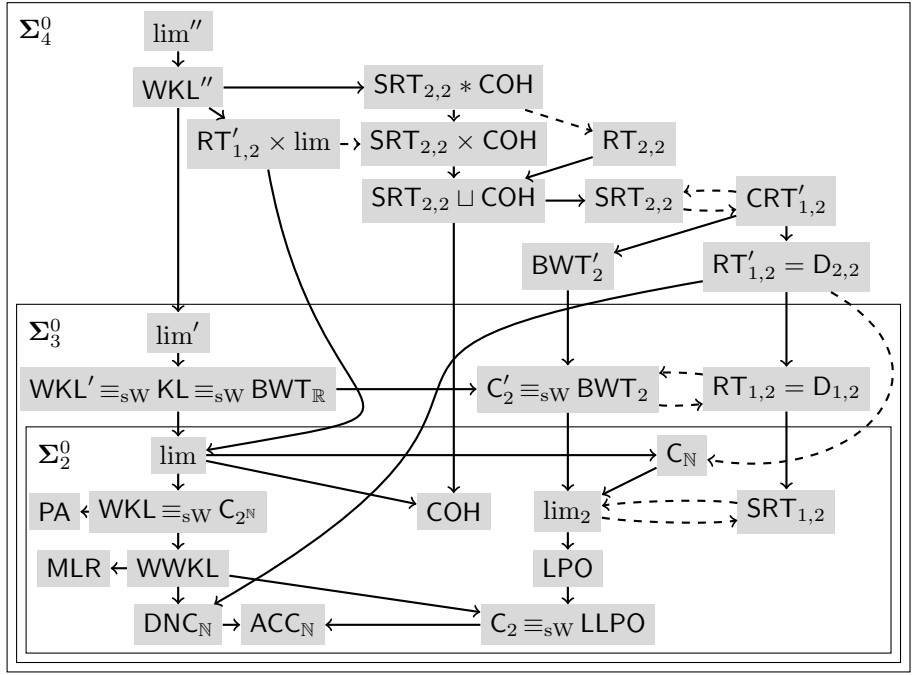

We start with two straightforward separation results that show that the jump of the cylindrification of $\RT_{1,k}$
(see Lemma~\ref{lem:jump-cylindrification}) is incomparable with $\RT_{2,2}$. 

\begin{proposition}
\label{prop:RT2*RT1xlim}
$\RT_{2,2}\nleqW\RT_{1,\IN}'\times\lim$ and $\RT_{1,2}'\times\lim\nleqW\RT_{2,\IN}$.
\end{proposition}
\begin{proof}
For every computable input $\RT_{1,\IN}'\times\lim$ clearly has a limit computable output,
since for a limit computable coloring $c:[\IN]^1\to\IN$ there is always some $k\in\IN$ such that
$A:=\{i\in\IN:c(i)=k\}$ is infinite. Such an $A$ is clearly limit computable and $A\in\RT_{1,\IN}(c)$. 
However, $\RT_{2,2}$ has computable inputs
$c:[\IN]^2\to2$, which have no limit computable outputs $A\in\RT_{2,2}(c)$ by Theorem~\ref{thm:Jockusch}.
Hence, $\RT_{2,2}\nleqW\RT_{1,\IN}'\times\lim$.

While $\lim\leqW\RT_{1,2}'\times\lim$ holds obviously, we obtain $\lim\nleqW\RT_{2,\IN}$ by Corollary~\ref{cor:limit-avoidance1}.
Hence, $\RT_{1,2}'\times\lim\nleqW\RT_{2,\IN}$.
\end{proof}

This result yields the following corollary.

\begin{corollary}
\label{cor:RT12xlim-RT22}
$\RT_{1,2}'\times\lim\nW\RT_{2,2}$.
\end{corollary}

Since we have $\SRT_{2,\IN}\leqW\RT_{1,\IN}*\lim\equivW\RT_{1,\IN}'\times\lim$ by Corollary~\ref{cor:SRT-RT-lim} and Lemma~\ref{lem:jump-cylindrification},
we also obtain the following corollary of Proposition~\ref{prop:RT2*RT1xlim}

\begin{corollary}
\label{cor:RT22-SRT2N}
$\RT_{2,2}\nleqW\SRT_{2,\IN}$ and hence, in particular, $\RT_{2,2}\nleqW\SRT_{2,2}$.
\end{corollary}

The corresponding non-uniform result in reverse mathematics was much harder to obtain and solved a longstanding
open question \cite{CSY14}.
The proof uses a non-standard model, whereas our result would follow from a separation with a standard model.

Among other things, we are going to discuss the relation of $\RT_{2,2}$ to the cohesiveness problem.
The  following proof is essentially the proof of Cholak, Jockusch and Slaman \cite[Theorem~12.5]{CJS01}.

\begin{proposition}[Cohesiveness]
\label{prop:COH-RT22}
$\COH\leqSW\RT_{2,2}$.
\end{proposition}
\begin{proof}
Let $(R_i)_i$ be a sequence of sets $R_i\In\IN$. We compute a sequence $(S_i)_i$ by $S_{2i}:=R_i$
and $S_{2i+1}:=\{i\}$ for all $i\in\IN$. We let $d(i,j):=\min\{k:\chi_{S_k}(i)\not=\chi_{S_k}(j)\}$.
The definition of $(S_i)_i$ ensures that $d$ is well-defined for all $i<j$, and it can be computed from $(S_i)_i$.
We compute a coloring $c:[\IN]^2\to2$ with
\[c\{i<j\}:=\left\{\begin{array}{ll}
  0 & \mbox{if $i\in S_{d(i,j)}$}\\
  1 & \mbox{otherwise}
\end{array}\right.,\]
and we consider an infinite homogeneous set $M\in\RT_{2,2}(c)$ and $i\in\IN$. 
We claim that $M$ is cohesive for $(S_i)_i$ and hence for $(R_i)_i$. 
Let us assume for a contradiction that $k$ is the smallest number such that
$M\cap S_k$ and $M\cap(\IN\setminus S_k)$ are infinite. 
Since $k$ is minimal, there is a number $m\in\IN$ such that $d(i,j)\geq k$ for all $i,j\geq m$ in $M$.
There are also sufficiently large $i_0,i_1,i_2\geq m$ in $M$ such that $i_0<i_1<i_2$, $\chi_{S_k}(i_0)\not=\chi_{S_k}(i_1)$ and $\chi_{S_k}(i_1)\not=\chi_{S_k}(i_2)$.
This implies $d(i_0,i_1)=k=d(i_1,i_2)$, which in turn yields $c\{i_0,i_1\}\not=c\{i_1,i_2\}$, in contradiction to the homogeneity of $M$.
Hence, $M$ is cohesive for $(R_i)_i$, and we obtain $\COH\leqSW\RT_{2,2}$.
\end{proof}

We obtain $\RT_{2,2}\nleqW\COH$ since $\COH\leqSW\lim$ by \cite[Proposition~12.10]{BHK15},
but on the other hand $\RT_{2,2}\nleqW\lim'$ by Corollary~\ref{cor:SRT-limits}.
We obtain the following corollary.

\begin{corollary}
\label{cor:COH-RT22-SRT22}
$\SRT_{2,2}\sqcup\COH\leqW\RT_{2,2}\leqW\SRT_{2,2}*\COH\leqSW\WKL''$.
\end{corollary}
\begin{proof}
By Proposition~\ref{prop:COH-RT22}, Lemma~\ref{lem:basic-reductions} and since $\sqcup$ is the supremum
with respect to $\leqW$ we have
$\SRT_{2,2}\sqcup\COH\leqW\RT_{2,2}$.
The reduction $\RT_{2,2}\leqW\SRT_{2,2,}*\COH$ follows from Proposition~\ref{prop:RT-SRT-COH}.
The last mentioned reduction in the corollary 
follows from Corollary~\ref{cor:upper-bound} and Facts~\ref{fact:WKL-lim-COH} and \ref{fact:WKL-BWT} since
\[\SRT_{2,2}*\COH\leqW\WKL''*\COH\leqW\WKL'*\lim*\COH\equivW \WKL'*\WKL'\equivW \WKL''.\qedhere\]
\end{proof}

Given the fact that 
$\SRT_{2,2}\sqcup\COH\leqSW\SRT_{2,2}\times\COH\leqW\SRT_{2,2}*\COH$,
it would be desirable to clarify the relation of $\RT_{2,2}$ to all the three mentioned problems.

\begin{question}
\label{quest:SRT22-COH-RT22}
How does $\RT_{2,2}$ exactly relate to $\SRT_{2,2}\sqcup\COH$, $\SRT_{2,2}\times\COH$ and $\SRT_{2,2}*\COH$?
\end{question}

We can at least say something. We have $\COH\leqW\lim$ by \cite[Proposition~12.10]{BHK15} and by Corollary~\ref{cor:SRT-RT-lim}
and Lemma~\ref{lem:jump-cylindrification} 
we know that $\SRT_{2,2}\leqW\RT_{1,2}'\times\lim$.
Since $\lim$ is idempotent we obtain the following corollary.

\begin{corollary}
\label{cor:SRT22-COH-RT12-lim}
$\SRT_{2,2}\times\COH\leqW\RT_{1,2}'\times\lim$.
\end{corollary} 

With Proposition~\ref{prop:RT2*RT1xlim} we arrive at the following conclusion.

\begin{corollary}
\label{cor:RT22-SRT22-COH}
$\RT_{2,2}\nleqW\SRT_{2,2}\times\COH$.
\end{corollary}

We also mention that Corollary~\ref{cor:products} and Proposition~\ref{prop:COH-RT22} imply the following.

\begin{corollary}
$\SRT_{2,2}\times\COH\leqW\RT_{2,4}$.
\end{corollary}

We recall that the problem $\RT_{1,2}'$ is also studied under the name $\D_2^2$ (see for instance \cite{CJS01,CLY10}).
We introduce the following notation, where we use lower indices again.

\begin{definition}
We define $\D_{n,k}:=\RT_{1,k}^{(n-1)}$ for all $n,k\geq1$.
\end{definition}

Dzhafarov proved that $\COH\nleqSW\D_{2,k}$ holds for all $k\in\IN$ \cite[Corollary~1.10]{Dzh15}.
In fact, his key theorem \cite[Theorem~1.5]{Dzh15} yields even the following
stronger result.

\begin{corollary}
\label{cor:COH-D}
$\COH\nleqSW\CRT_{1,k}^{(m)}$ for all $k,m\in\IN$. 
\end{corollary}

In a subsequent paper \cite{Dzh16} Dzhafarov proved the following result.

\begin{theorem}
\label{thm:COH-SRT22}
$\COH\nleqW\SRT_{2,+}$ and hence, in particular, $\COH\nleqW\SRT_{2,2}$.
\end{theorem}

Since K\H{o}nig's lemma $\KL$ is often discussed in the context of $\RT_{2,2}$, we mention it briefly
in passing. By $\KL:\In\Tr_\IN\mto\IN^\IN$ we denote the multi-valued function that is defined on all infinite finitely branching
trees $T\In\IN^*$ and such that $\KL(T)=[T]$ is the set of infinite paths of $T$.
The set $\Tr_\IN$ of trees $T\In\IN^*$ is represented via characteristic functions of such trees as usually. 
We can also consider a variant $\KL_+$ of $\KL$, where the set of trees 
is represented by positive information only, i.e., by an enumeration of the corresponding tree $T$. 
The following theorem is essentially based on results from \cite{BGM12}. 
If $n\in\IN$, then we write $\widehat{n}\in\IN^\IN$ for the constant sequence with value $n$.

\begin{theorem}
\label{thm:KL}
$\KL\equivSW\KL_+\equivSW\WKL'$.
\end{theorem}
\begin{proof}
It is clear that $\KL\leqSW\KL_+$.
By \cite[Corollaries~11.6 and 11.7]{BGM12} we have that $\WKL'\equivSW\BWT_{\IN^\IN}$.

It is easy to see that $\KL_+\leqSW\BWT_{\IN^\IN}$:
Given a finitely branching tree $T\In\IN^*$ by an enumeration $T=\{w_n:n\in\IN\}$
we just compute a sequence $(x_n)_n$ in $\IN^\IN$ with $x_n:=w_n\widehat{0}$.
Since $T$ is finitely branching, the set $\overline{\{x_n:n\in\IN\}}$ is compact, and hence $(x_n)_n$ has cluster points. 
It is clear that all cluster points of $(x_n)_n$ are infinite paths of $T$. This proves $\KL_+\leqSW\BWT_{\IN^\IN}$.

We now prove $\BWT_{\IN^\IN}\leqSW\KL_+$. Given a sequence $(x_n)_n$ that lies in a compact subset of $\IN^\IN$,
we enumerate a tree $T$ as follows. We start with the empty tree $T$ and in step $n=0,1,2,...$ we inspect $x_n$.
If $w\prefix x_n$ is the longest prefix of $x_n$ that was already enumerated into $T$, then we enumerate
$x_n|_k$ into $T$ in step $n$, where $k=|w|+1$. In this way a tree $T$ is enumerated such that $[T]$ is the
set of cluster points of $(x_n)_n$. The tree $T$ is a finitely branching tree since $(x_n)_n$ lies within a compact set.

Finally, we prove $\KL_+\leqSW\KL$. Given a finitely branching tree $T\In\IN^*$ by an enumeration $T=\{w_n:n\in\IN\}$,
we compute a finitely branching tree $S\In\IN^*$ (by determining its characteristic function) 
such that $\pr_1([S])=[T]$, where 
\[\pr_1(\langle n_0,k_0\rangle,\langle n_1,k_1\rangle,\langle n_2,k_2\rangle...)=(n_0,n_1,n_2,...)\in\IN^\IN\]
for all $n_i,k_i\in\IN$.
Since $\pr_1$ is computable, we obtain $\KL_+\leqSW\KL$ in this way.
We create the tree $S$ in stages $n=0,1,2,...$ and in stage $n$ we inspect the $n$--th word $w_n$ that is enumerated into $T$.
In this stage we decide that a word of the form $u_{w,v}:=\langle w_n(0),v(0)\rangle\langle w_n(1),v(1)\rangle...\langle w_n(k-1),v(k-1)\rangle$ with $k=|w_n|$ and all its prefixes
belong to $S$, where $v$ is the lexicographically smallest word $v\in\IN^k$ such that $u_{w,v}$ and all its prefixes have not
yet been decided to belong to $\IN^*\setminus S$. Additionally, we decide $w\in\IN^*\setminus S$ for all non-empty words $w$ that have not yet been decided to belong to $S$ and such that 
$|w|\leq n$ and $w(i)\leq n$ for all $i\leq n-1$. 
On the one hand, this algorithm ensures that the tree $S$ is decided for larger and larger blocks of size $n$,
and on the other hand, the construction guarantees that the resulting tree $S$ is finitely branching: on each level
$i$ of $T$ only finitely many different values can occur and at some stage $n$ of the construction all words $w$ of length
$i$ in $T$ have been considered and corresponding words $u_{w,v}$ have been added to $S$.
The fact that we always choose the lexicographically smallest $v$ ensures that no new strings of length $i$ are added to $S$ after stage $n$. Finally, the construction also guarantees $\pr_1([S])=[T]$ as promised. 
\end{proof}

Now Theorem~\ref{thm:KL} and Corollary~\ref{cor:delayed-parallelization} yield the following result,
which was independently proved in \cite[Corollary~2.3]{HJ16} by a very different method.

\begin{corollary}
\label{cor:KL}
$\KL\leqW\RT_{3,2}$.
\end{corollary}

We note that by Corollary~\ref{cor:LLPO-RT} the reduction cannot be replaced by a strong one
since $\C_2\leqSW\WKL\leqSW\KL$.
Likewise, one obtains $\WWKL\nleqSW\RT_{n,k}$ for all $n,k\geq1$ since $\C_2\leqSW\WWKL$.
We note that by Corollary~\ref{cor:limit-avoidance} we obtain that Corollary~\ref{cor:KL} is optimal in the following sense.

\begin{corollary}
\label{cor:KL-RT22}
$\KL\nleqW\RT_{2,\IN}$.
\end{corollary}

However, we note the related Questions~\ref{quest:WKL-SRT32} and \ref{quest:WKL-RT22}.
Liu proved the following theorem and lemma \cite[Theorem~1.5 and proof of Corollary~1.6]{Liu12}.

\begin{theorem}[Liu 2012~\cite{Liu12}]
\label{thm:Liu}
For any set $C$ not of $\PA$--degree and any $A\In\IN$ there exists an infinite subset $G$ of $A$ or $\IN\setminus A$
such that $G\oplus C$ is also not of $\PA$--degree.
\end{theorem}
 
\begin{lemma}[Liu 2012~\cite{Liu12}]
\label{lem:Liu}
For any set $C$ not of $\PA$--degree and any uniformly $C$--computable sequence $(C_i)_i$ there exists a set 
$G$ which is cohesive for $(C_i)_i$ and such that $G\oplus C$ is also not of $\PA$--degree.
\end{lemma}

We consider Peano arithmetic as the problem $\PA:\DD\mto\DD,{\bf b}\mapsto\{{\bf a}:{\bf a}\gg{\bf b}\}$ in the Weihrauch lattice, where $\DD$ denotes the set of Turing degrees represented by their members and ${\bf a}\gg{\bf b}$
expresses the property that ${\bf a}$ is a $\PA$--degree relative to ${\bf b}$.
From Liu's Theorem~\ref{thm:Liu} we can directly derive $\PA\nleqW\D_{n,2}$ for all $n\geq1$, and we also obtain the following corollary.

\begin{corollary}
\label{cor:PA}
$\PA\nleqW\CRT_{1,2}^{(n-1)}$ for all $n\geq1$.
\end{corollary}

In particular, this implies $\PA\nleqW\SRT_{2,2}$, since $\SRT_{2,2}\equivW\CRT_{1,2}'$ by Theorem~\ref{thm:CRT-SRT}.
Using Liu's Lemma~\ref{lem:Liu} this can be strengthened as follows.

\begin{corollary}
\label{cor:PA-RT22}
$\PA\nleqW\SRT_{2,2}*\COH$.
\end{corollary}

Since $\RT_{2,2}\leqW\SRT_{2,2}*\COH$ by Proposition~\ref{prop:RT-SRT-COH}, this implies  $\PA\nleqW\RT_{2,2}$.
Since $\PA\lW\WKL$ \cite[Theorem~5.5, Corollary~6.4]{BHK15}, we also get the following conclusion.

\begin{corollary}
\label{cor:WKL}
$\WKL\nleqW\RT_{2,2}$.
\end{corollary}

Likewise, one should be able to use the methods provided by Liu \cite{Liu15} to show that $\MLR\nleqW\RT_{2,2}$ (where $\MLR$ denotes the problem to find a point that is Martin-L\"of random relative to the input) and
hence $\WWKL\nleqW\RT_{2,2}$.
It would also be interesting to find out whether a variant of the above ideas can be used to answer Question~\ref{quest:WKL-RT22}.

By $\ConC_X$ we denote the restriction of $\C_X$ to connected subsets. 
We now want to discuss the intermediate value theorem $\IVT$, which we do not define here, but we
identify it with $\ConC_{[0,1]}$ (see \cite[Corollary~4.10]{BLRMP16a} and \cite[Theorem~6.2]{BG11a}, 
where $\ConC_{[0,1]}$ appears under the names $\ConC_1$ and $\C_I$, respectively).
$\ConC_{[0,1]}$ can be seen as a one-dimensional version of the Brouwer fixed point theorem 
and thus of the equivalence class of $\WKL$ (see \cite{BLRMP16a}).
Hence, from the uniform perspective it is an obvious question to ask how $\IVT$ is related to Ramsey's theorem. 
In order to approach this question we provide a new upper bound on $\IVT$ which is of independent interest.
By ${\CL_X:\In X^\IN\mto X}$ we denote the {\em cluster point problem} of $X$, which is the problem
to find a cluster point of the given input sequence (this is an extension of $\BWT_X$ since we do not
demand that the input sequence lies within a compact set, but only that it admits a cluster point).
By \cite[Theorem~9.4]{BGM12} we have $\CL_X\equivSW\C_X'$ for every computable metric space $X$.

\begin{proposition}
\label{prop:IVT}
$\IVT\leqW\C_\IN'$.
\end{proposition}
\begin{proof}
We have $\IVT\equivW\ConC_{[0,1]}\equivW\B_I$ \cite[Proposition~3.6 and Theorem~6.2]{BG11a},
where $\B_I:\In\IR_<\times\IR_>\mto\IR,(x,y)\mapsto\{z:x\leq z\leq y\}$ with $\dom(\B_I):=\{(x,y):x\leq y\}$
denotes the boundedness problem introduced in \cite{BG11a}
and $\C_\IN'\equivW\CL_\IN$~\cite[Theorem~9.4]{BGM12}.
Here $\IR_<$ and $\IR_>$ denote the real numbers represented as supremum and infimum of sequences
of rational numbers, respectively.
The aim is to prove $\B_I\leqW\CL_\IN$.
Hence, we can assume that we have two monotone sequences $(a_n)_n$ and $(b_n)_n$ of rational numbers in $[0,1]$ 
with $a_n\leq a_{n+1}<b_{n+1}\leq b_n$ for all $n\in\IN$, and the goal is to find some $x\in[0,1]$ with $a:=\sup_{n\in\IN}a_n\leq x\leq\inf_{n\in\IN}b_n=:b$.
Now we produce a sequence $(c_n)_n$ of natural numbers in stages $n=0,1,2,...$ as follows, where we use a variable $k\in\IN$ that is initially $k=0$.
If at stage $n$ we find that $|b_n-a_n|<2^{-k-1}$, then we let $k:=k+1$, and we choose $c_n:=0$.
Otherwise, we choose some value $c_n\not=0$ such that $a_n<\overline{c_n}<b_n$, where $\overline{c_n}$ is the rational number with code $c_n$.
If possible, we choose $c_n$ such that it is identical to the previous such number chosen, and if that is impossible, then we chose $c_n$ with $\overline{c_n}:=a_n+\frac{b_n-a_n}{2}$. 
Now there are two possible cases. If $a=b$, then the sequence $(c_n)_n$ contains infinitely many zeros and at most one other number $c$ occurs infinitely
often in $(c_n)_n$ and it must be such that $a=b=\overline{c}$. Otherwise, if $a\not=b$, then the sequence  $(c_n)_n$ contains only finitely many zeros,
and only exactly one number $c$ different from zero occurs infinitely often. This number satisfies $a\leq \overline{c}\leq b$.
The result of $\CL_\IN(c_n)_n$ is one number $c$ that occurs infinitely often in $(c_n)_n$. If $c=0$, then we can compute $x:=a=b$ with the help of the input
$(a_n)_n$, $(b_n)_n$. If $c\not=0$, then $x:=\overline{c}$ is a suitable result. 
\end{proof}

Propositions~\ref{prop:IVT}, \ref{prop:KN-CN} and Theorem~\ref{thm:jumps-compact-choice}, 
which we are going to prove in 
Section~\ref{sec:boundedness}, yield the reduction $\IVT\leqW\SRT_{2,\IN}$.
Ludovic Patey (personal communication) improved this result by showing that $\IVT\leqW\RT_{2,2}$.
This answered an open question in an earlier version of this article. 

$\RT_{2,2}$ and $\SRT_{2,2}$ cannot be Weihrauch reducible to any of $\PA,\MLR$ or $\COH$ since the latter are reducible to $\lim$
and the former are not.
However, we can even say more. 
We recall that a problem $f:\In X\mto Y$ is called {\em densely realized}, if the set $\{F(p):F\vdash f\}$ 
is dense in the domain of the representation of $Y$ for all $p$ that are names for an input in the domain of $f$,
which intuitively means that a name of a result can start with any arbitrary prefix (that is a prefix of a valid name).
Since it is easy to see that $\PA,\MLR$ and $\COH$ are all densely realized, it follows by \cite[Proposition~4.3]{BHK15}
that $\PA^{(m)},\MLR^{(m)}$ and $\COH^{(m)}$ are all indiscriminative. In particular we obtain the following corollary.

\begin{corollary}
\label{cor:RT22-PA-MLR-COH}
$\SRT_{1,2}\nleqW\PA^{(m)}$, $\SRT_{1,2}\nleqW\MLR^{(m)}$ and $\SRT_{1,2}\nleqW\COH^{(m)}$ for all $m\in\IN$.
\end{corollary}

We mention that it was noted by Hirschfeldt and Jockusch \cite[Figure~6]{HJ16} that the following
corollary follows from \cite[Theorem~2.3]{HJK+08}.

\begin{corollary}[Diagonally non-computable functions]
\label{cor:DNC}
$\DNC_\IN\leqSW\RT_{1,2}'=\D_{2,2}$.
\end{corollary}

More information on the uniform content and relations between the problems  $\DNC,\PA,\WKL,\WWKL,\MLR,\COH$ and 
other problems can be found in \cite{BHK15}.

\section{Separation Techniques for Jumps}
\label{sec:separation}

In this section we plan to prove some further separation results related to $\RT_{2,2}$ 
(that justify some of the missing arrows in the diagram in Figure~\ref{fig:diagram-RT22}).
For this purpose we collect some classical results that yield separation techniques
for jumps. Often continuity arguments can be used for separation results. 
In the case of reductions to jumps of the form $f\leqW g'$, such continuity arguments
are not always applicable, since jumps $g'$ implicitly involve limits. 
However, continuity arguments can occasionally be replaced by arguments 
based on the existence of continuity points, which are still available for limit computable functions.
The classical techniques that we recall here
concern $\SO{2}$--measurable functions, which form the topological counterpart of limit computable functions.
Baire proved that the set of points of continuity of every $\SO{2}$--measurable function is a
comeager $G_\delta$--set under certain conditions. We recall that a set is
called {\em comeager} if it contains a countable intersection of dense open sets. 
The following is result is \cite[Theorem~24.14]{Kec95}.

\begin{proposition}[Baire~\cite{Kec95}]
\label{prop:Baire}
Let $X,Y$ be metric spaces, and let $Y$ additionally be separable.
If $f:X\to Y$ is $\SO{2}$--measurable, then the set of points of continuity
of $f$ is a comeager $G_\delta$--set in $X$.
\end{proposition}

We recall that a topological space $X$ is called a {\em Baire space}\footnote{We distinguish between {\em a} Baire space $X$ in general and {\em the} Baire space $\IN^\IN$, in particular, which is also an instance of a Baire space.}, 
if every comeager set $C\In X$ is dense in the space $X$. Hence we get the following corollary.

\begin{corollary}
\label{cor:Baire}
Let $X$ and $Y$ be metric spaces, where $X$ is additionally a Baire space and $Y$ is separable.
If $f:X\to Y$ is $\SO{2}$--measurable, then $f|_U$ has a point of continuity for every non-empty 
open set $U\In X$.
\end{corollary}

By the Baire category theorem \cite[Theorem~8.4]{Kec95}, every complete separable metric space is a Baire space.
Moreover, closed subspaces of complete separable metric spaces are complete again. 
Hence Corollary~\ref{cor:Baire} easily implies the direction ``(1) $\TO$ (2)'' of the following 
characterization of $\SO{2}$--measurable functions \cite[Theorem~24.15]{Kec95}.

\begin{theorem}[Baire characterization theorem~\cite{Kec95}]
\label{thm:Baire}
Let $X,Y$ be separable metric spaces, and let $X$ additionally be complete. 
For every function $f:X\to Y$ the following conditions are equivalent to each other:
\begin{enumerate}
\item $f$ is $\SO{2}$--measurable,
\item $f|_A$ has a point of continuity for every non-empty closed set $A\In X$.
\end{enumerate}
\end{theorem}

We note that while Theorem~\ref{thm:Baire} yields a stronger conclusion than Corollary~\ref{cor:Baire},
it also requires stronger conditions on the domain $X$. In certain situations, we will be dealing with
Baire spaces $X^\IN$ which are not Polish spaces, and hence we will have to take recourse to Corollary~\ref{cor:Baire}.

We recall \cite[Proposition~9.1]{Bra05} that the limit map $\lim$ is a prototype of a $\SO{2}$--measurable function
(relatively to its domain of convergent sequences), and hence Theorem~\ref{thm:Baire} can be
used as a separation tool for certain reductions to jumps, since limit computable functions are closed under composition
with continuous functions.

\begin{fact} 
\label{fact:lim}
$\lim:\In\IN^\IN\to\IN^\IN$ is $\SO{2}$--measurable relatively to its domain.
\end{fact}

For the following result we use the Baire characterization theorem~\ref{thm:Baire} as a separation technique.
For partial functions $f:\In X\to Y$ we use the notation $f(A):={\{y\in Y:(\exists x\in A)\; f(x)=y\}}$ 
for the image of $A$ under $f$.

\begin{theorem}
\label{thm:BWT-RTxlim}
$\BWT_{k+1}\nleqW\RT_{1,k}'\times\lim$ for all $k\geq1$.
\end{theorem}
\begin{proof}
Since $\RT_{1,k}'\times\lim$ is a cylinder, it suffices to prove $\BWT_{k+1}\nleqSW\RT_{1,k}'\times\lim$.
Let us assume for a contradiction that $\BWT_{k+1}\leqSW\RT_{1,k}'\times\lim$,
i.e., there are computable functions $H,K$ such that $HGK$ is a realizer
of $\BWT_{k+1}$ whenever $G$ is a realizer of $\RT_{1,k}'\times\lim=(\RT_{1,k}\times\id)\circ(\lim\times\lim)$.
Since $\BWT_{k+1}$ is total, it follows
that $\langle K_1,K_2\rangle:=\langle\lim\times\lim\rangle\circ K$ is total too,
and this function is $\SO{2}$--measurable by Fact~\ref{fact:Borel}.
We construct a number of points $p_0,...,p_k\in\IN^\IN$ inductively. 
By the Theorem~\ref{thm:Baire} there is a point $p_0\in\IN^\IN$ 
of continuity of $\langle K_1,K_2\rangle$.
Let $\langle s_0,r_0\rangle:=\langle K_1(p_0),K_2(p_0)\rangle$,
let $q_0\in\RT_{1,k}(s_0)$ be a maximal homogeneous set, and let $c_0\leq k-1$ be the
corresponding color.
Let $k_0\leq k$ be such that $H\langle q_0,r_0\rangle=k_0$.
By continuity of $H$ there exists a prefix $w_0\prefix\langle q_0,r_0\rangle$
such that $H(w_0\IN^\IN)=\{k_0\}$. 
We can assume that $w_0=\langle x_0,y_0\rangle$ is of even length $2n_0$.
Since $\langle K_1,K_2\rangle$ is continuous at $p_0$, there is a $v_0\prefix p_0$ such that
$K_1(v_0\IN^\IN)\In s_0|_{n_0}\IN^\IN$ and $K_2(v_0\IN^\IN)\In y_0\IN^\IN$.
Now we consider $N_0:=\{0,...,k\}\setminus\{k_0\}$ and the closed set $A_0:=v_0N_0^\IN$.
Again by Theorem~\ref{thm:Baire} there exists a point $p_1\in A_0$ of continuity
of $\langle K_1,K_2\rangle|_{A_0}$. 
Let $\langle s_1,r_1\rangle:=\langle K_1(p_1),K_2(p_1)\rangle$,
let $q_1\in\RT_{1,k}(s_1)$ be a maximal homogeneous set, and let $c_1\leq k-1$ be the corresponding color.
Let $k_1\leq k$ be such that $H\langle q_1,r_1\rangle = k_1$.
Since $p_1\in A_0$ contains $k_0$ at most finitely many times, it follows
that $k_1\not=k_0$. By continuity of $H$ there exists a prefix $w_1\prefix \langle q_1,r_1\rangle$
such that $H(w_1\IN^\IN)=\{k_1\}$.
Again we can assume that $w_1=\langle x_1,y_1\rangle$ is of even length $2n_1$,
and additionally we can assume $n_1>n_0$.
Since $v_0\prefix p_1$ we have $y_0\prefix K_2(p_1)=r_1$, and we obtain
\[s_1|_{n_0}=K_1(p_1)|_{n_0}=K_1(p_0)|_{n_0}=s_0|_{n_0}.\]
If $c_1=c_0$, i.e., if the maximal homogeneous sets $q_0$ and $q_1$ of $s_0$ and $s_1$, respectively,
have the same color, then $x_0=q_0|_{n_0}=q_1|_{n_0}$ follows due to the maximality of $q_0$ and $q_1$ 
and hence $w_0\prefix\langle q_1,r_1\rangle$, 
which implies $H\langle q_1,r_1\rangle=k_0$ and hence $k_0=k_1$. 
Since $k_0\not=k_1$, we can conclude that $c_0\not=c_1$.
Now we continue as before:
since $\langle K_1,K_2\rangle|_{A_0}$ is continuous at $p_1$, there is a $v_1\prefix p_1$ such that
$K_1(v_1\IN^\IN)\In s_1|_{n_1}\IN^\IN$ and $K_2(v_1\IN^\IN)\In y_1\IN^\IN$.
Now we consider $N_1:=\{0,...,k\}\setminus\{k_0,k_1\}$ and the closed set $A_1:=v_1N_1^\IN$.
Again by Theorem~\ref{thm:Baire} there exists a point $p_2\in A_1$ of continuity
of $\langle K_1,K_2\rangle_{A_1}$. In this case we eventually obtain colors $k_2\leq k,c_2\leq k-1$ such that
$k_2\not\in\{k_0,k_1\}$ and $c_2\not\in\{c_0,c_1\}$.
This construction can be repeated $k$ times since there are $k+1$ colors in $\{0,...,k\}$ until
$p_0,...,p_k$ and $c_0,...,c_k$ are determined. However, the construction yields that the
$c_0,...,c_k$ have to be pairwise different, which is a contradiction since there are only $k$ colors $c\in\{0,...,k-1\}$ available.  
\end{proof}

We get the following immediate corollary with the help of Lemma~\ref{lem:jump-cylindrification} and Corollary~\ref{cor:SRT-RT-lim}.

\begin{corollary}
\label{cor:BWTk-SRT2k}
$\BWT_{k+1}\nleqW\SRT_{2,k}$ for all $k\geq1$.
\end{corollary}

Since $\BWT_{k+1}\leqW\RT_{1,k+1}\leqW\SRT_{2,k+1}$ by Proposition~\ref{prop:bottom} and Lemma~\ref{lem:increasing-size},  
we obtain also the following separation result as a consequence of Corollary~\ref{cor:BWTk-SRT2k}.

\begin{corollary}
\label{cor:SRT-k-k+1}
$\SRT_{2,k+1}\nleqW\SRT_{2,k}$ for all $k\geq1$.
\end{corollary}

As next separation result we want to prove $\BWT_2'\nleqW\RT_ {1,2}'$.
To this end we are going to use Corollary~\ref{cor:Baire} as a separation tool.
It is not too difficult to see that for every topological space $X$ that is endowed with the
discrete topology, the product space $X^\IN$ is a Baire space. 
We will apply this observation to the particular case of the set 
\[X_2:=\{p\in2^\IN:(\exists k)(\forall n\geq k)\;p(n)=p(k)\}\]
of convergent sequences of zeros and ones. 
On the other hand, $X_2$ is also endowed with the subspace topology that it
inherits from the ordinary product topology on $\IN^\IN$.
In both cases, we assume that $X_2^\IN$ is equipped with the respective product topology.
In order to distinguish both topologies, we speak about ``discrete continuity'' of a map $f:X_2^\IN\to Y$,
if $X_2$ is endowed with the discrete topology and about continuity if the usual subspace topology
of $\IN^\IN$ is used for $X_2$. In the latter case, the map 
\[h_2:X_2^\IN\to\IN^\IN,((x_{i,j})_i)_j\mapsto\langle(x_{0,j})_j,(x_{1,j})_j,(x_{2,j})_j,...\rangle,\]
is a computable embedding, i.e., $h_2$ as well as the partial inverses $h_2^{-1}$ are 
both computable and continuous. Intuitively, $h_2$ swaps rows and columns and applies the Cantor encoding $\langle ...\rangle$,
both being purely technical changes. Hence we will identify $X_2^\IN$ with (a subspace of) $\IN^\IN$ via $h_2$ in the following proof.
We note that $\BWT_2'$ formally has the same domain as $\BWT_2$, namely $2^\IN$.
But a realizer of $\BWT_2'$ has to use $\lim_{2^\IN}$ as input representation and hence it will have domain $h_2(X_2^\IN)$,
which we identify with $X_2^\IN$.
For $w\in\IN^n$ and $p\in\IN^\IN$ we use the notation
$w^\rightarrow p:=wp(n)p(n+1)p(n+2)...$ for the sequence that is obtained from $p$ by replacing the first
$n$ positions of $p$ by $w$. 

\begin{theorem}
\label{thm:BWT2j-RT12j}
$\BWT_2'\nleqW\RT_ {1,2}'$.
\end{theorem}
\begin{proof}
Let us assume for a contradiction that $\BWT_2'\leqW\RT_{1,2}'$. 
Then there are computable functions $H,K$ such that $H\langle \id,GK\rangle$ is a realizer of $\BWT_2'$
whenever $G$ is a realizer of $\RT_{1,2}'$. 
We tacitly identify $X_2^\IN$ with a subspace of $\IN^\IN$ via $h_2$.
Since $\BWT_2':2^\IN\mto\{0,1\}$ is total, it follows that $\lim K:X_2^\IN\to\IN^\IN$ is also total
and $\SO{2}$--measurable. 
If we endow $X_2$ with the discrete topology, then $X_2^\IN$ is a Baire space, and by Corollary~\ref{cor:Baire} there
exists a point of discrete continuity $x\in X_2^\IN$ of $\lim K$.
Let $c:=\lim K(x)$ be the corresponding coloring. 
Then there is some maximal infinite homogeneous set $q\in\RT_{1,2}'(c)$, and there is a realizer $G$
of $\RT_{1,2}'$ that maps $K(x)$ to $q$. 
By continuity of $H$ there is some $m\in\IN$ such that $H\langle x|_mX_2^\IN,q|_m\IN^\IN\rangle=\{e\}$ for some $e\in\{0,1\}$.
Without loss of generality, we assume $e=0$.
Hence, by discrete continuity of $\lim K$ at $x$, there is some $k>m$ such that $\lim K(x|_kX_2^\IN)\In c|_m2^\IN$.

We construct some $y=(x|_k,z_0,z_1,z_2...)\in X_2^\IN$ such that all $z_i\in X_2$ converge to $1$, and the coloring $c_y:=\lim K(y)$ has maximal
homogeneous sets $q_0,q_1\in\RT_{1,2}(c_y)$ of different colors. This yields a contradiction: since $x|_k\prefix y$, we obtain $c|_m\prefix c_y$, 
and hence $q|_m$ must be prefix of either $q_0$ or $q_1$, say $q_0$, and this implies $H\langle y,q_0\rangle=e=0$, whereas the correct result would have to be $1$,
since all the $z_i$ converge to $1$.
We describe the inductive construction of $y$ in even and odd stages $i\in\IN$.

Stage $0$: Let $y_0:=x|_k\widehat{0}\,\widehat{0}...$, $c_0:=\lim K(y_0)$, and let $r_0\in\RT_{1,2}(c_0)$. 
Then by continuity of $H$ 
there is some $m_0>k$ such that $H\langle x|_{k}w_{m_0}X_2^\IN,r_0|_{m_0}\IN^\IN\rangle=\{0\}$
with $w_m:=(0^m\widehat{1})^m=(0^m\widehat{1},0^m\widehat{1},...,0^m\widehat{1})\in X_2^m$ for all $m\in\IN$.

Stage $1$: Let $y_1:=x|_kw_{m_0}\widehat{1}\,\widehat{1}...$, $c_1:=\lim K(y_1)$, and let $r_1\in\RT_{1,2}(c_1)$. 
Then also ${0^{m_0}}^\rightarrow r_1\in\RT_{1,2}(c_1)$, and by continuity of $H$ there
is some $m_1>k+m_0$ such that $H\langle y_1|_{m_1}X_2^\IN,{0^{m_0}}^\rightarrow r_1|_{m_1}\IN^\IN\rangle=\{1\}$.

We now assume that $i\geq1$ and that $y_0,...,y_{2i-1}\in X_2^\IN$ and $m_0,...,m_{2i+1}\in\IN$ have already been constructed.

Stage $2i$: Let $y_{2i}:=y_{2i-1}|_{m_{2i-1}}\widehat{0}\,\widehat{0}...$, $c_{2i}:=\lim K(y_{2i})$, and let $r_{2i}\in\RT_{1,2}(c_{2i})$. 
Then by continuity of $H$ there is some $m_{2i}>m_{2i-1}$ such that 
\[H\langle y_{2i-1}|_{m_{2i-1}}w_{m_{2i}}X_2^\IN,{0^{m_{2i-1}}}^\rightarrow  r_{2i}|_{m_{2i}}\IN^\IN\rangle=\{0\}.\]

Stage $2i+1$: Let $y_{2i+1}:=y_{2i-1}|_{m_{2i-1}}w_{m_{2i}}\widehat{1}\,\widehat{1}...$, $c_{2i+1}:=\lim K(y_{2i+1})$, and let $r_{2i+1}\in\RT_{1,2}(c_{2i+1})$.
Then also ${0^{m_{2i}}}^\rightarrow r_{2i+1}\in\RT_{1,2}(c_{2i+1})$, and by continuity of $H$ there is some $m_{2i+1}>m_{2i-1}+m_{2i}$ such that 
\[H\langle y_{2i+1}|_{m_{2i+1}}X_2^\IN,{0^{m_{2i}}}^\rightarrow r_{2i+1}|_{m_{2i+1}}\IN^\IN\rangle=\{1\}.\]

This construction yields a sequence $(y_i)_i$ in $X_2^\IN$ and a strictly increasing sequence
$(m_i)_i$ in $\IN$ such that $y_{2i-1}|_{m_{2i-1}}\prefix y_{2i+1}$ for all $i$.
Since all the blocks $w_m$ that are added to the initial $x|_k$ in the odd stages are of the form $w_m=(z_1,...,z_m)$,
with $z_i\in X_2$ that all converge to $1$, it follows that $(y_{2i+1})_i$ converges to a $y=(x|_k,z_0,z_1,z_2...)\in X_2^\IN$, with $z_i$ that all converge to $1$. 
Let $c_y:=\lim K(y)$ be the corresponding coloring.
We now prove that there are homogeneous sets in $\RT_{1,2}(c_y)$ of different color. 
We fix some maximal homogeneous set $r\in\RT_{1,2}(c_y)$.
It suffices to prove that there are infinitely many $i$ with $r(i)=0$.
The inductive proof follows the stages of the construction above and uses the corresponding objects constructed therein. 

Stage $0,1$:  We obtain $r_0|_{m_0}\not\In r$, since otherwise there is some $s\in\RT_{1,2}(c_y)$ such that $r_0|_{m_0}\prefix s$ and $H\langle y,s\rangle=1$.
Since $x|_kw_{m_0}\prefix y$ this contradicts the choice of $m_0$ in Stage 0.
In particular, we obtain that there is some $i$ with $0\leq i<m_0$ and $r(i)=0$.

Stages $2i,2i+1$: Likewise, ${0^{m_{2i-1}}}^\rightarrow r_{2i}|_{m_{2i}}\not\In r$, since otherwise there is some $s\in\RT_{1,2}(c_y)$ such that ${0^{m_{2i-1}}}^\rightarrow r_{2i}|_{m_{2i}}\prefix s$
and $H\langle y,s\rangle=1$.
Since $y_{2i-1}|_{m_{2i-1}}w_{m_{2i}}\prefix y$ this contradicts the choice of $m_{2i}$ in Stage $2i$.
In particular, there is some $i$ with $m_{2i-1}\leq i<m_{2i}$ and $r(i)=0$.

Altogether, this proves that there are infinitely many $i$ with $r(i)=0$, since $(m_i)_i$ is strictly increasing.
\end{proof}

By Corollary~\ref{cor:lower bounds with jumps} we have $\lim_\IN\leqW\RT_{1,2}'$.
On the other hand, we obtain the following corollary.

\begin{fact}
\label{fact:CN-BWT}
$\lim_\IN\nleqW\BWT_k^{(m)}$ for all $k,m\geq1$.
\end{fact}
\begin{proof}
Let us denote by $\U\C_\IN$ the restriction of $\C_\IN$ to singletons.
By \cite[Fact~3.2(1)]{BGM12} $\UC_\IN$ is a strong fractal, and $\C_\IN$ is obviously slim, 
which means that $\range(\U\C_\IN)=\range(\C_\IN)$.
Under these conditions \cite[Theorem~13.3]{BGM12} yields 
$\C_\IN\nleqW\BWT_{k}^{(m)}$ since $|\range(\C_\IN)|=|\IN|>k=|\range(\BWT_k^{(m)})|$.
By \cite[Proposition~3.8]{BGM12} we have $\lim_\IN\equivW\C_\IN$, which implies the claim.
\end{proof}

We could also replace the application of \cite[Theorem~13.3]{BGM12} in this proof by an 
application of Proposition~\ref{prop:choice-cardinality} that we prove below.
Now it follows with Fact~\ref{fact:CN-BWT}
that $\RT_{1,2}'\nleqW\BWT_2'$. Hence, together with Theorem~\ref{thm:BWT2j-RT12j}
we obtain the following corollary.

\begin{corollary}
\label{cor:BWT2j-RT12j}
$\BWT_2'\nW\RT_{1,2}'$.
\end{corollary}

By Proposition~\ref{prop:bottom} we have $\BWT_2\equivW\RT_{1,2}$, 
but Corollary~\ref{cor:BWT2j-RT12j} shows that
the situation is very different for strong Weihrauch reducibility.
With the help of \cite[Proposition~5.6(2)]{BGM12} we obtain the following.

\begin{corollary}
\label{cor:BWT2-RT12}
$\BWT_2\nSW\RT_{1,2}$.
\end{corollary}

It is clear that $\D_{2,2}=\RT_{1,2}'\leqSW\CRT_{1,2}'\equivW\SRT_{2,2}$ by Theorem~\ref{thm:CRT-SRT},
and by Theorem~\ref{thm:lower-bound} we also obtain $\BWT_2'\leqSW\CRT_{1,2}'$.
Now Corollary~\ref{cor:BWT2j-RT12j} also leads to the following
separation result which was independently proved by Dzhafarov \cite[Corollary~3.3]{Dzh16}.

\begin{corollary}
\label{cor:SRT22-D22}
$\SRT_{2,2}\nleqW\D_{2,2}$.
\end{corollary}

Altogether, we have now separated several of the degrees illustrated in the diagram
of Figure~\ref{fig:diagram-RT22}.

\section{Boundedness and Induction}
\label{sec:boundedness}

The purpose of this section is to collect some results on how Ramsey's theorem is related to 
boundedness and induction. In reverse mathematics there is a well-known strict hierarchy of
induction principles $\I\sO{n}$ and boundedness principles $\B\sO{n}$ (see \cite[Theorem~4.32]{Hir15}):
\[\B\sO{1}\leftarrow\I\sO{1}\leftarrow\B\sO{2}\leftarrow\I\sO{2}\leftarrow...\]
While these principles have no immediate interpretation as problems in the Weihrauch lattice,
there are equivalent formulations of these principles that do have such interpretations.
For instance, it is well-known that $\I\sO{n}$ is equivalent to the least number principle $\L\pO{n}$ over a very weak system \cite[Theorem~2.4]{HP93}.
The least number principle $\L\pO{1}$ can be interpreted as the the following problem:
\[\min:\In\IN^\IN\to\IN,p\mapsto\min\{n\in\IN:(\forall k\in\IN)\;p(k)\not=n\}.\]
Equivalently, one could also consider the map $\min:\In\AA_-(\IN)\to\IN,A\mapsto\min(A)$ that maps
every non-empty closed set $A\In\IN$ (given by an enumeration of its complement) to its minimum. 
The $n$--th jump $\min^{(n)}$ now clearly corresponds to $\L\pO{n+1}$ for all $n\in\IN$.
The following result is easy to obtain.

\begin{proposition}[Least number principle]
\label{prop:least}
$\C_\IN\equivSW\min$.
\end{proposition}
\begin{proof}
$\C_\IN\leqSW\min$ is obvious, since given an enumeration $p$ of a the complement of a non-empty set $A\In\IN$,
we have $\min(p)\in A$. 
We also have $\min\leqSW\C_\IN$: Given a sequence $p\in\IN^\IN$ we start enumerating the complement 
of a set $A\In\IN$ while we inspect the sequence $p$. The algorithm proceeds in stages $i=0,1,2,...$.
In stage $i$ we let $k_i:=\min\{n\in\IN:(\forall j\leq i)\;p(j)\not=n\}$ and we remove all $\langle n,m\rangle$ from $A$ for $n,m\leq i$ and $m\not=k_i$. Hence, the final set $A$ satisfies 
\[A\In\{\langle n,k\rangle\in\IN:n\in\IN,k=\min(\IN\setminus\range(p))\}\]
and given a number $\langle n,k\rangle\in A$ we can easily extract $k=\min(\IN\setminus\range(p))$. 
Hence, $\min\leqSW\C_\IN$.
\end{proof}

Since jumps preserve strong Weihrauch equivalences we obtain $\C_\IN^{(n)}\equivSW\min^{(n)}$ for all $n\in\IN$, and hence
$\C_\IN^{(n)}$ can be seen as a counterpart of $\I\sO{n+1}$ for all $n\in\IN$. 

Likewise, the boundedness principle $\B\sO{n}$ has been characterized with the help of the pigeonhole principle and the
regularity principle $\R\sO{n}$. More precisely, it is known that $\B\sO{n+2}$ is equivalent to $\R\sO{n+1}$ for all $n\in\IN$
over a very weak system~\cite[Theorem~2.23]{HP93}. 
Now $\BWT_\IN$ can be seen as a counterpart of $\R\sO{1}$ since the computational goal of $\BWT_\IN$ is to find
a number that appears infinitely often in a bounded sequence. More generally, $\R\sO{n+1}$ corresponds to $\BWT_\IN^{(n)}$ for all $n\in\IN$.
We recall that $\K_\IN$ denotes {\em compact choice} on $\IN$, 
which can be defined as closed choice $\C_\IN$, except that additionally an upper bound for the input set is provided (see \cite{BGM12} for a precise definition). 
Another characterization of $\K_\IN$ is $\K_\IN\equivSW\C_2^*$ (which holds by \cite[Proposition~10.9]{BGM12} where $\LLPO\equivSW\C_2$)
and we obtain $\K_\IN'\equivSW\BWT_\IN$ by \cite[Corollary~11.10]{BGM12}.
Hence, it is justified to say that $\K_\IN^{(n+1)}$ corresponds to $\B\sO{n+2}$ for all $n\in\IN$. 

Analogously to the above implication chain for $\I\sO{n}$ and $\B\sO{n}$ we obtain
\[\K_\IN\lW\C_\IN\lW\K_\IN'\lW\C_\IN'\lW\K_\IN''\lW\C_\IN''\lW....\]
This reduction chain is based on the following observation.

\begin{proposition}
\label{prop:KN-CN}
$\K_\IN^{(n)}\lSW\C_\IN^{(n)}\lSW\K_\IN^{(n+1)}$ and $\K_\IN^{(n)}\lW\C_\IN^{(n)}\lW\K_\IN^{(n+1)}$ for  $n\in\IN$.
\end{proposition}
\begin{proof}
We obtain $\K_\IN\leqSW\C_\IN\equivSW\lim_\IN\leqSW\BWT_\IN\equivSW\K_\IN'$.
Here $\C_\IN\equivSW\lim_\IN$ holds by \cite[Proposition~3.8]{BGM12}.
The claimed reductions follow since jumps are monotone with respect to $\leqSW$.
The separation results follow from  Fact~\ref{fact:WKL-BWT}~(5) since we have
$\widehat{\C_\IN}\equivSW\widehat{\lim_\IN}\equivSW\lim$ (see also \cite[Example~3.10]{BBP12}), 
$\widehat{\K_\IN}\equivSW\WKL$ and since parallelization commutes with jumps. 
Here $\widehat{\K_\IN}\equivSW\WKL$ holds
since we obtain that $\C_2\leqSW\C_2^*\equivSW\K_\IN\leqSW\widehat{\C_2}\equivSW\WKL$ and 
parallelization is a closure operator.
\end{proof}

In reverse mathematics a considerable amount of work has been spent in order to calibrate Ramsey's theorem for pairs according
to the above boundedness and induction principles. 
Hence, it is natural to ask how it compares uniformly to choice principles.
We start with some easy observations, and as a preparation we prove the following result.
The proof is a simple variation of the proof of \cite[Theorem~13.3]{BGM12}.

\begin{proposition}[Choice and cardinality]
\label{prop:choice-cardinality}
Let $f:\In X\mto\IN$ and $X\In\IN$ be such that $\C_X\leqW f$. Then $|X|\leq|\range(f)|$.
\end{proposition}
\begin{proof}
Let us assume that $\C_X\leqW f$. We use the representation $\psi_-$ of closed subsets of $X$.
Then there are computable $H,K$ such that $H\langle\id,GK\rangle\vdash\C_X$ for all $G\vdash f$.
We fix some realizer $G\vdash f$.
For simplicity and without loss of generality we assume that $G$ and $H$ have target space $\IN$.
We consider the following claim: for each $i\in\IN$ with $|X|>i$ 
there exist
\begin{enumerate}
\item $k_i\in X\setminus\{k_0,...,k_{i-1}\}$,
\item $n_i\in\range(f)\setminus\{n_0,...,n_{i-1}\}$,
\item $p_i$, a name of a closed subset $A_i\In X$, 
\item $w_i\prefix p_i$, 
\end{enumerate}
such that $w_{i-1}\prefix w_i$, $GK(p_i)=n_i$ and $H\langle w_i\IN^\IN,n_i\rangle=k_i$.
Let $w_{-1}$ be the empty word. 
We prove this claim by induction on $i$. 
Let us assume that $|X|>0$, and let $p_0$ be a name of $A_0:=X$.
Then $k_0:=H\langle p_0,GK(p_0)\rangle\in A_0$ and $n_0:=GK(p_0)$.
By continuity of $H$ there is some $w_0\prefix p_0$ such that 
such that $H\langle w_0\IN^\IN,n_0\rangle=k_0$. 
Let us now assume that $|X|>1$, and let $p_1$ be a name of $A_1:=X\setminus\{k_0\}$ with $w_0\prefix p_1$.
Then $k_1:=H\langle p_1,GK(p_1)\rangle\in A_1$ and $n_1:=GK(p_1)$.
Since $k_1\not=k_0$, we obtain $n_0\not=n_1$ since $H\langle w_0\IN^\IN,n_0\rangle=k_0$.
By continuity of $H$ there is some $w_1\prefix p_1$ with $w_0\prefix w_1$
such that $H\langle w_1\IN^\IN,n_1\rangle=k_1$. 
The proof can now continue inductively as above which proves the claim.
The claim implies $|X|\leq|\range(f)|$.
\end{proof}

From this result we can conclude the following observation.

\begin{proposition}[Finite Choice]
\label{prop:finite-choice}
$\C_k\leqW\SRT_{1,k}$ and $\C_{k+1}\nleqW\RT_{1,k}$ for all $k\geq1$.
\end{proposition}
\begin{proof}
By Lemma~\ref{prop:bottom} we have $\lim_k\equivW\SRT_{1,k}$, hence the first reduction follows from \cite[Corollary~13.8]{BGM12} and can easily been proved directly.
By Lemma~\ref{prop:bottom} we have $\BWT_k\equivW\RT_{1,k}$, and hence the second claim follows from Proposition~\ref{prop:choice-cardinality}.
\end{proof}

Since $\K_\IN\equivW\C_2^*$ by \cite[Proposition~10.9]{BGM12} we get the following conclusion.

\begin{proposition}[Compact choice]
\label{prop:compact-choice}
$\K_\IN\leqW\SRT_{1,+}$ and $\K_\IN\nleqW\RT_{1,k}$ for all $k\geq1$.
\end{proposition}
\begin{proof}
Since $\C_2\leqW\SRT_{1,2}$, we obtain $\K_\IN\equivW\C_2^*\leqW\SRT_{1,2}^*\leqW\SRT_{1,+}$ by Corollary~\ref{cor:finite-parallelization}.
On the other hand, $\C_{k+1}\nleqW\RT_{1,k}$ and $\C_{k+1}\leqW\K_\IN$ implies $\K_\IN\nleqW\RT_{1,k}$ for all $k\geq1$.
\end{proof}

We obtain the following corollary.

\begin{corollary}[Jump of compact choice]
\label{cor:jump-compact-choice}
$\K_\IN'\equivW\RT_{1,\IN}$ and $\K_\IN'\nleqW\SRT_{2,+}$.
\end{corollary}
\begin{proof}
$\K_\IN'\equivW\RT_{1,\IN}$ follows from Proposition~\ref{prop:bottom} since $\K_\IN'\equivSW\BWT_\IN$.
Corollary~\ref{cor:BWTk-SRT2k} implies $\K_\IN'\nleqW\SRT_{2,k}$ for all $k\geq1$, and since $\K_\IN'$ is countably irreducible (as any jump is by \cite[Proposition~5.8]{BGM12})
we obtain $\K_\IN'\nleqW\SRT_{2,+}$.
\end{proof}

The equivalence $\K_\IN'\equivW\RT_{1,\IN}$ corresponds to a well-known theorem of Hirst~\cite{Hir87}, which says that $\RT^1_{<\infty}$ is equivalent to $\B\sO{2}$ over $\RCA_0$ (see also \cite[Theorem~6.81]{Hir15}),
whereas the second equivalence $\K_\IN'\nleqW\SRT_{2,+}$ shows that the reverse mathematics result \cite{CJS01} (see also \cite[Theorem~6.82]{Hir15})
that $\SRT_{2,2}$ proves $\RT_{<\infty}^1$ over $\RCA_0$ cannot be proved uniformly (for instance the proof presented in \cite[Theorem~6.82]{Hir15} contains a non-constructive case distinction).

Corollary~\ref{cor:jump-compact-choice} yields also one direction of the following corollary,
and the other direction follows since $\RT_{1,\IN}\leqW\lim'$, but $\SRT_{2,2}\nleqW\lim'$ and $\SRT_{2,\IN}\nleqW\lim'$  by Corollary~\ref{cor:SRT-limits}.

\begin{corollary}
\label{cor:SRT22-RT1N}
$\RT_{1,\IN}\nW\SRT_{2,2}$ and $\RT_{1,\IN}\nW\SRT_{2,+}$.
\end{corollary}

We can conclude the following results on $\C_\IN$ from earlier results.

\begin{proposition}[Closed choice]
\label{prop:closed-choice}
$\C_\IN\equivW\SRT_{1,\IN}$ and $\C_\IN\nleqW\RT_{1,+}$.
\end{proposition}
\begin{proof}
The first statement follows from Proposition~\ref{prop:bottom} since $\C_\IN\equivW\lim_\IN$ by \cite[Proposition~3.8]{BGM12}, and the second statement follows from the fact that $\C_\IN$ is countably irreducible by \cite[Fact~3.2]{BGM12} (since
every strong fractal is a fractal and hence countably irreducible),
but $\C_\IN\leqW\RT_{1,k}$ for some $k\geq1$ is impossible by Proposition~\ref{prop:compact-choice} since $\K_\IN\leqW\C_\IN$.
\end{proof}

Altogether, we have justified the way choice problems are displayed in Figure~\ref{fig:diagram-RTnk}.
Finally, we provide the following reduction that can be derived from earlier results.

\begin{theorem}[Jumps of compact choice]
\label{thm:jumps-compact-choice}
$\K_\IN^{(n)}\leqW\SRT_{n,\IN}$ for all $n\geq2$.
\end{theorem}
\begin{proof}
By Propositions~\ref{prop:bottom} and \ref{prop:jumps} and since jumps are monotone with respect to $\leqSW$ we obtain
$\K_\IN^{(n)}\equivSW\BWT_\IN^{(n-1)}\leqSW\CRT_{1,\IN}^{(n-1)}\leqSW\CSRT_{n,\IN}\leqW\SRT_{n,\IN}$.
\end{proof}

The special case for $n=2$ can be seen as the uniform
version of a theorem of Cholak, Jockusch and Slaman \cite{CJS01}, see also \cite[Theorem~6.89]{Hir15},
which states that $\SRT_{<\infty}^2$ proves $\B\Sigma^0_3$ over $\RCA_0$.
In light of Corollary~\ref{cor:jump-compact-choice} and Theorem~\ref{thm:jumps-compact-choice} there is quite some gap in between $\SRT_{2,+}$ and $\SRT_{2,\IN}$.
The diagram summarizes our calibration of choice problems by Ramsey's theorem.

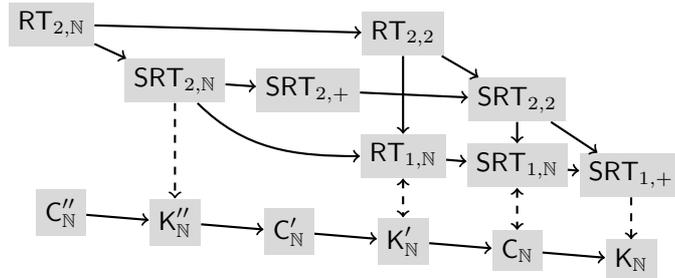
\begin{figure}[htb]
\begin{tikzpicture}[scale=.5,auto=left,every node/.style={fill=black!15}]

\def\rvdots{\raisebox{1mm}[\height][\depth]{$\huge\vdots$}};

  % Nodes
  \node (SRT2N) at (3,4.75) {$\SRT_{2,\IN}$};
  \node (RT2N) at (-0.25,6.25) {$\RT_{2,\IN}$};
  \node (RT1N) at (9,2.75) {$\RT_{1,\IN}$};
  \node (RT22) at (9,6) {$\RT_{2,2}$};
  \node (SRT22) at (12,4.25) {$\SRT_{2,2}$};
  \node (SRT1N) at (12,2.5) {$\SRT_{1,\IN}$};
  \node (SRT2P) at (6.5,4.5) {$\SRT_{2,+}$};
  \node (SRT1P) at (15,2.25) {$\SRT_{1,+}$};
 %  \node (CNSSS) at (-6,1.75) {$\C_\IN'''$};
 % \node (KNSSS) at (-3,1.5) {$\K_\IN'''$};
  \node (CNSS) at (0,1.25) {$\C_\IN''$};
  \node (KNSS) at (3,1) {$\K_\IN''$};
  \node (CNS) at (6,0.75) {$\C_\IN'$};
  \node (KNS) at (9,0.5) {$\K_\IN'$};
  \node (CN) at (12,0.25) {$\C_\IN$};
  \node (KN) at (15,0) {$\K_\IN$};

%   % Straight lines
   \foreach \from/\to in {
   RT2N/RT22,
   RT22/SRT22,
   SRT22/SRT1P,
   RT2N/SRT2N,
   SRT2N/SRT2P,
   SRT2P/SRT22,
   RT22/RT1N,
   RT1N/SRT1N,
   SRT1N/SRT1P,
   SRT22/SRT1N,
 %  CNSSS/KNSSS,
 %  KNSSS/CNSS,
   CNSS/KNSS,
   KNSS/CNS,
   CNS/KNS,
   KNS/CN,
   CN/KN}
   \draw [->,thick] (\from) -- (\to);

%   % Dashed lines
   \foreach \from/\to in {
   SRT2N/KNSS,
   SRT1P/KN}
   \draw [->,thick,dashed] (\from) -- (\to);

%   % Dashed double lines
 \draw[<->,thick,dashed] (RT1N) to (KNS);
 \draw[<->,thick,dashed] (SRT1N) to (CN);
% 

%   % Special curved lines
  \draw [->,thick,looseness=1] (SRT2N) to [out=315,in=180] (RT1N);

\end{tikzpicture}
  
\ \\[-0.5cm]
\caption{Closed and compact choice on natural numbers calibrated with Ramsey's theorem in the Weihrauch lattice.}
\label{fig:diagram-KC}
\end{figure}

Further questions could be studied along these lines. We mention the following question.
The first part of this question is related to \cite[Theorem~6.85]{Hir15} and the second part
to \cite[Open Question~6.92]{Hir15}.

\begin{question}[Jump of closed choice]
\label{quest:jump-closed-choice}
Does $\C_\IN'\leqW\RT_{2,2}$ hold? Does $\C_\IN''\leqW\RT_{2,\IN}$ hold?
\end{question}

In light of \cite[Theorem~6.85]{Hir15} a positive answer to the first part of this question seems
unlikely. 
Another natural question is how cluster point problems are related to Ramsey's theorem.
We can say at least something.

\begin{corollary}[Cluster point problem]
\label{cor:CR-RT33}
$\CL_\IR\leqW\RT_{3,3}$.
\end{corollary}
\begin{proof}
We recall that $\CL_\IN\equivW\C_\IN'$ and $\CL_{2^\IN}\equivW\C_{2^\IN}'\equivW\WKL'$ by \cite[Theorem~9.4]{BGM12}.
By \cite[Proposition~9.15]{BGM12} we have $\CL_\IR\equivW\C_\IN'\times\CL_{2^\IN}\equivW\C_\IN'\times\WKL'$.
Hence, we obtain with Proposition~\ref{prop:KN-CN}, Theorem~\ref{thm:jumps-compact-choice} and Corollary~\ref{cor:delayed-parallelization}
\[\CL_\IR\equivW\C_\IN'\times\WKL'\leqW\K_\IN''\times\WKL'\leqW\RT_{2,\IN}\times\RT_{3,2}\leqW\RT_{3,3}.\]
The last mentioned reduction holds by Theorem~\ref{thm:products}.
\end{proof}

However, it is not immediately clear whether the following holds. 

\begin{question}
\label{quest:CR-RT32}
$\CL_\IR\leqW\RT_{3,2}$?
\end{question}

\section{Conclusion}

We have studied the uniform computational content of Ramsey's theorem in the Weihrauch lattice,
and we have clarified many aspects of Ramsey's theorem in this context. 
Key results are the lower bound provided in Theorem~\ref{thm:lower-bound}, the theorems on products (Theorem~\ref{thm:products})
and parallelization (Theorem~\ref{thm:delayed-parallelization}), as well as the theorem on jumps (Theorem~\ref{thm:CRT-SRT}) and the upper
bounds in Corollary~\ref{cor:upper-bound} derived from it.
From this tool box of key results (together with the squashing theorem, Theorem~\ref{thm:squashing}) 
we were able to derive a number of interesting consequences, such as the characterization of the 
parallelization of Ramsey's theorem in Corollary~\ref{cor:parallelization} and the effect of 
increasing numbers of colors in Theorem~\ref{thm:increasing-colors}.
The separation tools provided in Section~\ref{sec:separation} have led to some further clarity. 
A number of important questions regarding the uniform behavior of Ramsey's theorem were left open.
Hopefully, some future study will shed further light on this question.

\section{Acknowledgments}

We would like to thank Steve Simpson and Ludovic Patey for helpful comments on an earlier version of this article
and the anonymous referee for his or her very careful proof reading that helped us to improve the presentation of the article.

\bibliographystyle{plain}
%\bibliography{C:/Users/Vasco/Dropbox/Bibliography/lit}
\bibliography{C:/Users/vbrattka/Dropbox/Bibliography/lit}

\end{document}